\documentclass[a4paper,10pt]{amsart}
\usepackage[utf8]{inputenc}
 \usepackage{setspace}
\usepackage[cmtip,arrow]{xy}
\usepackage{lscape}
\usepackage{latexsym}
\usepackage[activeacute,english]{babel}
\usepackage{graphicx}
\usepackage{amsmath}
\usepackage{amsfonts}

\usepackage{pst-node}

\usepackage{amsthm}
\usepackage{amssymb}
\usepackage{amscd}
\usepackage{mathrsfs}
\usepackage{pb-diagram,pb-xy}
\xyoption{all}
\usepackage{graphics}
\usepackage{stmaryrd}
\usepackage{yfonts,bbm}
\usepackage{tikz-cd}

\newtheorem{theorem}{Theorem}[section]
\newtheorem{lemma}[theorem]{Lemma}
\newtheorem{proposition}[theorem]{Proposition}
\newtheorem{corollary}[theorem]{Corollary}
\newtheorem{definition}[theorem]{Definition}
\newtheorem{example}[theorem]{Example}

\newtheorem{Remark}[theorem]{Remark}

\usepackage[square,sort,comma,numbers]{natbib}
\usepackage{hyperref}
\hypersetup{colorlinks=true,linkcolor=blue,citecolor=red,linktocpage=true}
\usepackage{url}
\usepackage{relsize}

\newcommand{\A}{\mathcal{A}}
\newcommand{\B}{\mathcal{B}}
\newcommand{\C}{\mathcal{C}}
\newcommand{\D}{\mathcal{D}}
\newcommand{\T}{\mathcal{T}}

\newcommand{\Z}{\mathbb{Z}}
\newcommand{\Ho}{\mathrm{Hom}}

\begin{document}
\title{A Recollement of differential graded categories}

\author[M. L. S. Sandoval-Miranda, V. Santiago, E. O. Velasco]{Martha Lizbeth Shaid Sandoval-Miranda,\\ Valente Santiago-Vargas,\\ Edgar Omar Velasco-P\'aez}

\thanks{The authors are grateful to the project PAPIIT-Universidad Nacional Aut\'onoma de M\'exico IN100520}
\dedicatory{}

\begin{abstract}
	In this paper, we  prove that given a differential graded category $\mathcal{C}$ and $\mathcal{B}$ a full differential graded subcategory  closed under coproducts, there is a canonical recollement 
	$$\xymatrix{ \mathrm{DgMod}(\mathcal{C}/ \mathcal{I}_{\mathcal{B}})  \ar@<0ex>[r]& \mathrm{DgMod}(\mathcal{C}) \ar@<0ex>[r] \ar@<-2ex>[l]  \ar@<2ex>[l]& \mathrm{DgMod}(\mathcal{B}) \ar@<-2ex>[l] \ar@<2ex>[l]}$$
	of differential graded categories, for which we use enriched categories tools. 
	We continue  the study of differential graded  triangular matrix categories $\Lambda=\left[ \begin{smallmatrix}
		\mathcal{T} & 0 \\ 
		M & \mathcal{U}
	\end{smallmatrix}\right]$ as initiated in \cite{PaezSandovalSan}. We show that given a recollement between functor dg-categories we can induce a new recollement between differential graded triangular matrix  categories, this is a generalization of a result given by Chen and Zheng in \cite[Theorem 4.4]{Chen}.
\end{abstract}
\maketitle

\section{Introduction}\label{sec:1}

The notion of recollement of triangulated categories was introduced first by Beilinson, Bernstein, and Deligne in \cite{Beilinson1}. In the context of abelian categories, recollements were studied by Franjou and Pirashvili in \cite{Franjou}, motivated by the work of MacPherson-Vilonen in derived category of perverse sheaves (see \cite{Macpherson}).  Recollements of triangulated (abelian) categories can be viewed as ''exact sequences'' of triangulated (abelian) categories, which describe the middle term by a subcategory and a quotient category. It should be noted that these two recollement situations are now widely used in the study of representation theory and algebraic geometry.\\
On the other hand, rings of the form $\left[ \begin{smallmatrix}
T & 0 \\ 
M & U
\end{smallmatrix}\right]$ where $T$ and $U$ are rings and $M$ is a $T,U$-bimodule have appeared often in the study of the representation theory of artin rings and algebras (see for example \cite{Auslander1}, \cite{Green2}). Such a rings appear naturally in the study of homomorphic images of hereditary artin algebras.\\

The differential graded categories (dg-categories) and their dg-modules have played a fundamental role in mathematics for a long time. In the last century,  dg-categories were the primary tool to study the matrix problems related to the representation theory of algebras. In addition, the development of their study has led to Drozd's proof of the tame-wild dichotomy (see, for example: \cite{Yuriy2}, \cite{Mark}). \\

In recent decades, the relevance of dg-categories comes from a fundamental result due to B. Keller,  which asserts that each compactly generated algebraic triangulated  category is equivalent to the derived category of certain small dg-category, see \cite[Theorem 4.3]{Keller}.\\

We now give a brief description of the contents on this paper.

In Section 2,  we present some basic notations, mainly taken from \cite{Keller} and \cite{Keller2} (see also \cite{Yeku1} for a textbook), which will be used all throughout these paper.\\
In Section 3, we consider $F:\mathcal{A}\longrightarrow \mathcal{C}$ a dg-functor and the induced funtor $F_{\ast}:\mathrm{DgMod}(\mathcal{C})\longrightarrow \mathrm{DgMod}(\mathcal{A})$. The goal of all this section is to show that $F_{\ast}$ has a left and a right adjoint (see Proposition \ref{dgadjuninduced}). The left and right adjoints of $F_{\ast}$ are computed as left and right Kan extensions respectively. So, in order to see that  first we recall the notions of ends and coends in the Subsection \ref{endcoend}, then we study weighted limits and colimits in the dg-context in the Subsection \ref{weightedlimi},  in the Subsection \ref{Kanexten} we will give the notion of Kan extensions, and finally in the Subsection \ref{PropertyKan}, we recall the universal property that has the left and right Kan extensions  which help us to see that $F_{\ast}$ has left and right adjoints.\\
In Section 4, we consider $\mathcal{C}$ be a dg-category with finite coproducts  and $\mathcal{B}$ a full differential graded subcategory  closed under coproducts, we prove that  
there is a  dg ideal $\mathcal{I}_{\mathcal{B}}$ and that there is a recollement
$$\xymatrix{ \mathrm{DgMod}(\mathcal{C}/ \mathcal{I}_{\mathcal{B}})  \ar@<0ex>[r]& \mathrm{DgMod}(\mathcal{C}) \ar@<0ex>[r] \ar@<-2ex>[l]  \ar@<2ex>[l]& \mathrm{DgMod}(\mathcal{B}) \ar@<-2ex>[l] \ar@<2ex>[l]}$$ (see Theorem \ref{dg-recollement}).\\
In Section 5, we consider  two  dg categories $\mathcal{U}$, $\mathcal{T}$ and $M \in \mathrm{DgMod}(\mathcal{U}\otimes \mathcal{T}^{op})$ and we consider the  differential graded matrix category $ \Lambda = \left[ \begin{smallmatrix}
\mathcal{T} & 0 \\ 
M & \mathcal{U}
\end{smallmatrix}\right]$ defined in \cite{PaezSandovalSan}. The main result of this Section is a generalization of the result in \cite[Theorem 4.4]{Chen} and  \cite[Theorem 3.10]{LGOS2}, that is,  given a recollement between dg-functor categories we can induce a recollement between dg-modules over certain triangular matrix dg-categories $\mathrm{DgMod}\big( \left[ \begin{smallmatrix}
\mathcal{T} & 0 \\ 
M & \mathcal{U}
\end{smallmatrix}\right]\big)$ (see Theorem \ref{Recoll1}).

\section{Preliminaries}
In this section we present some basic notations of dg-categories, mainly taken from \cite{Keller} and \cite{Keller2} (see also \cite{Yeku1} for a textbook), which will be used all throughout these paper. Throughout this section $K$ will be an arbitrary commutative ring. Recall that a differential graded (dg) $K$-module is a graded $K$-module $M= \bigoplus_{n \in \Z} M^{n}$, together with a graded $K$-linear morphism $d_{M}: M\longrightarrow M$ of degree $+1$ such that $d_{M}\circ d_{M} = 0$. The category which will be  denoted by $\mathrm{DgMod}(K)$ has as objects the dg $K$-modules; and
given $M$ and $N$ two dg $K$-modules the set of morphisms from $M$ to $N$ in the category $\mathrm{DgMod}(K)$ is by definition the following graded $K$-module: 
$$\Ho_{\mathrm{DgMod}(K)}(M, N ):= \bigoplus_{n \in \Z} \Ho_{\mathrm{DgMod}(K)}^{n}(M, N ),$$
where $ \Ho_{\mathrm{DgMod}(K)}^{n}(M, N )$ consists of the graded $K$-linear morphisms $f:M \longrightarrow N$ of degree $n$, i.e., such that $f(M^{i})\subseteq N^{i+n}$ for all $i \in \Z$.
Moreover, each space of morphisms $\Ho_{\mathrm{DgMod}(K)}(M, N) $ has a structure of differential graded  $K$-module given by the  differential $d: \Ho_{\mathrm{DgMod}(K)}(M, N) \longrightarrow \Ho_{\mathrm{DgMod}(K)}(M, N)$, which is a graded $K$-linear morphism of degree $1$ such that $d \circ d = 0$ and is defined by the rule:
\begin{equation}\label{Homdgestructure}
d(\alpha) = d_{N} \circ \alpha -(-1)^{|\alpha|} \alpha \circ d_{M}
\end{equation}
where $|?|$ denotes the degree, whenever $\alpha$ is a homogeneous element
of the set $\Ho_{\mathrm{DgMod}(K)}(M, N)$.\\
Now,  if $M$ and $N$ are dg $K$-modules, the tensor product $M \otimes N := M \otimes_{K} N$ also becomes an object of $\mathrm{DgMod}(K)$ where the grading is given by $(M \otimes N )^{n} = \bigoplus_{i+j=n} M^{i}\otimes_{K} N^{j}$ and the differential $d_{M \otimes N} : M \otimes N \longrightarrow M \otimes N$ by the rule
\begin{equation} \label{tendgestructure}
d_{M \otimes N} (m \otimes n) = d_{M}(m) \otimes n + (-1)^{|m|}m \otimes d_{N}(n),
\end{equation}
for all homogeneous elements $m \in M$ and $n \in N$. All throughout this paper, we use the unadorned symbol $\otimes$ to denote $\otimes_{K}$.\\

\begin{definition}
 A differential graded category or a $\textbf{dg-category}$ is a $K$-category $\C $ such that:
 \begin{itemize}
 \item[(a)] $\Ho_{\mathcal{C}}(X,Y)$ $\in \mathrm{DgMod}(K)$   for all $ X,Y \in \mathcal{C}$.
 \item[(b)] The composition function  
 $$\theta_{X,Y,Z}: \Ho_{\mathcal{C}}(Y,Z) \otimes_{K} \Ho_{\mathcal{C}}(X,Y) \longrightarrow \Ho_{\mathcal{C}}(X,Z)$$
is a morphism of degree zero in $\mathrm{DgMod}(K)$  which commutes with the differentials.
 \end{itemize}
 \end{definition}
 
\begin{example}
\begin{enumerate}
\item[(a)] A dg-category with only one object can be identified with a dg-algebra, i.e., a graded $K$-algebra with a differential $d$ such that it satisfies Leibniz's rule $d(fg) = d(f)g + (-1)^{|f|}fd(g)$
for f, g homogeneous elements. Conversely, each dg-algebra $B$ can be viewed as a dg-category with only one object.
\item[(b)] The category $\mathrm{DgMod}(K)$ is the main example of a dg-category.
\end{enumerate}
\end{example}

If $\A$ is a dg-category, then the $\textbf{opposite dg-category}$ $\A^{op}$ has the same class of objects as $\A$ and the differential on morphisms $d:\A^{op}(A, B)= \A(B, A)\longrightarrow \A(B, A) = \A^{op}(A, B)$ is the same as in $\A$, that is, $d_{\A^{op}(A, B)}(f^{op}):=(d_{\A(B,A)}(f))^{op}$. The composition of morphisms in $\A^{op}$ is given as $\beta^{op} \circ  \alpha^{op} = (-1)^{|\alpha||\beta|}(\alpha \circ \beta)^{op}$, where we use the superscript $op$ to emphasize that a morphism is viewed in $\A^{op}$.\\
Analogously to the tensor product of $K$-modules we recall the definition of the tensor product of two differential graded categories.
\begin{definition}
Let $\A,$ $\B$ be dg-categories. The $\textbf{tensor product of dg-categories}$, which we denote by $\A \otimes \B$ is defined as follows:
 \begin{itemize}
  \item[(a)] Obj($\mathcal{A} \otimes \B$):= $Obj(\A) \times Obj(\B)$.
  \item[(b)]$\Ho_{\A \otimes \B}((X,Y),(X',Y')):= \Ho_{\A}(X,X') \otimes \Ho_{\B}(Y, Y')$ for all  $(X,Y),$ $ (X',Y') \in Obj(\A \otimes \B)$.
\end{itemize}
The composition 
  $\circ:\Ho_{\A \otimes \B}((X',Y'),(X'',Y'')) \times \Ho_{\A \otimes \B}((X,Y),$ $(X',Y')) \longrightarrow \Ho_{\A \otimes \B}((X,Y),(X'',Y''))$ is given as:
  $$g \circ f= (\alpha_{2} \otimes \beta_{2}) \circ (\alpha_{1} \otimes \beta_{1})= (-1)^{|\beta_{2}| |\alpha_{1}|} (\alpha_{2}\alpha_{1}) \otimes (\beta_{2}\beta_{1}),$$ 
 for each $f=\alpha_{1} \otimes \beta_{1}$ and $g=\alpha_{2} \otimes \beta_{2}$ with
$\alpha_{1}\in \Ho_{\A}(X,X')$, $\alpha_{2}\in \Ho_{\A}(X',X'')$, $\beta_{1} \in \Ho_{\B}(Y,Y')$ and  $\beta_{2} \in \Ho_{\B} (Y',Y'')$ homogenous elements.
\end{definition}

\begin{definition}
Let  $\A$, $\B$ be two  dg-categories.  A  $\textbf{graded functor}$ $F: \A \longrightarrow \B$   is a  $K$-linear functor such that $F(\Ho_{\A}^{n} (A, A'))\subseteq \Ho_{\B}^{n} (F(A), F(A'))$  for  each  $n \in \mathbb{Z}$ and  $A, A' \in Obj(\A)$.
 \end{definition}
 
 \begin{definition}
 Let $\A$, $\B$ two dg-categories. A  $\textbf{differential graded functor}$ or $\textbf{dg-functor}$ is a  graded functor  $F: \A \longrightarrow \B$  which commutes with  the differential, i.e, $F(d_{\A}(\alpha))= d_{\B}(F(\alpha))$ for all homogeneous morphism $\alpha$ in $\A$.
 \end{definition}
 One of the most important examples of functor requires the following lemma.
  \begin{lemma} \label{AxB dg funtor} 
Let $\A$, $\B$  and $\C$ be  dg-categories and  let $F: \A \otimes \B \longrightarrow \C$ be  an assignment on objects $(A,B) \rightsquigarrow F(A,B)$ and an assignment  on homogeneous morphism  $\alpha \otimes \beta \rightsquigarrow F(\alpha \otimes \beta)$ such that $|F(\alpha \otimes \beta)|= |\alpha|+ |\beta|$.Then the following conditions are equivalent.
\begin{itemize}
\item[(a)]The given mapping defines a dg-functor $F: \mathcal{A} \otimes \mathcal{B} \longrightarrow \mathcal{C}$.
\item[(b)]The following conditions are satisfied:
\begin{itemize}
\item[(b1)]For each fixed object  $A \in \A$, the assignment $B \rightsquigarrow F(A,B)$ on objects and $\beta \rightsquigarrow F(1_{A} \otimes \beta)$ on morphisms, define a dg-functor $G: \mathcal{B} \longrightarrow \mathcal{C}$.
\item[(b2)] For each fixed object $B \in \B$, the assignment $A \rightsquigarrow F(A,B)$ on objects and  $\alpha \rightsquigarrow F(\alpha \otimes 1_{B})$ on morphisms define a dg-functor $H: \mathcal{A} \longrightarrow \mathcal{C}$.
\item[(b3)]For all homogeneous morphisms $\alpha: A \longrightarrow A'$ and  $\beta: B \longrightarrow B'$, in  $\A$ and $\B$, respectively, there is the equality
$$(-1)^{|\alpha| |\beta |} F(1_{A'} \otimes \beta) \circ F(\alpha \otimes 1_{B})= F(\alpha \otimes \beta)= F(\alpha \otimes 1_{B'})\circ F(1_{A} \otimes \beta).$$
\end{itemize}
\end{itemize}
\end{lemma}
\begin{proof}
See \cite[Lemma 1.1]{Saorin}.
\end{proof}

\begin{example}\label{Hom(A,-) dg funtor}
\begin{enumerate}
\item [(a)] 
Let $\A$ be an  dg-category and  $\mathrm{DgMod}(K)$  the category  of differential graded $K$-modules. Let $A \in \A$ be then $\Ho_{\A}(A,-): \A \longrightarrow \mathrm{DgMod}(K)$ is a dg-functor where $\Ho_{A}(A,f)\!:\!\! \Ho_{\A}(A,A') \rightarrow \Ho_{\A} (A,B')$ is given by  $\Ho_{\A}(A,f)(j)= f \circ j$ for $f$, $j$  homogeneous elements.

\item [(b)] Let $\A^{op}$ a dg-category and let $\mathrm{DgMod}(K)$ the category  of differential graded $K$-modules. If $A \in \A^{op}$, then $\Ho_{\A^{op}}(-,A): \A^{op} \longrightarrow \mathrm{DgMod}(K)$  is a dg-functor where  $\Ho_{\A^{op}}(f^{op},A): \Ho_{\A^{op}}(B',A) \longrightarrow \Ho_{\A^{op}} (A',A)$ is given by $\Ho_{\mathcal{A}^{op}}(f,A)(j)= (-1)^{|f||j|} j \circ f$ for each  pair  $f$, $j$ of homogeneous elements.
\end{enumerate}
\end{example}

Using the Proposition \ref{AxB dg funtor} and examples  \ref{Hom(A,-) dg funtor} it follows that $\Ho_{\mathcal{A}}(-,-): \A^{op} \otimes \A  \longrightarrow \mathrm{DgMod}(K)$ is a dg-functor given for each  $(A,A') \in Obj(\A^{op} \otimes \A)$  as  $\Ho_{\mathcal{A}}(-,-)(A,A'):= \Ho_{\mathcal{A}}(A,A')$ and for   $\alpha: A \longrightarrow B$ and $ \alpha': A' \longrightarrow B'$  homogeneous morphism  in $\A$  we have that $\Ho_{\A}(\alpha^{op} \otimes \alpha')$ is defined  as follows:  $\Ho_{\A}(\alpha^{op} \otimes \alpha')(f):= (-1)^{|\alpha| (|\alpha'|+|f|)} \alpha' \circ f \circ \alpha$ for each homogeneous element $f \in \Ho_{\A}(B,A')$. 

Let us remember that the $\textbf{category of small dg-categories}$, which we denote by $\mathbf{dg}$-$\mathbf{Cat_{K}}$ is the category that has as objects all the small dg-categories, and as morphisms the dg-functors between them. To avoid sizing problems, it is requested that  $\mathrm{dg}$-$\mathrm{Cat}_{K}$ consists only of small dg-categories, fixed to a given universe, similar to $\mathrm{Cat}$ being the category of all small categories. To be able to work with dg-categories we require the notion of natural transformation in the dg context.

Now, we consider $\C$ and $\D$ two small dg-categories. In order to show that  the category of covariant dg-functors from $\mathcal{C}$ to $\mathcal{D}$ possesses a natural structure of differentially graded category, we need the following definition.

 \begin{definition}\label{dgnaturaltrans}
Let $\mathcal{C}$ and $\mathcal{D}$ be  small dg-categories, let $F,G: \mathcal{C} \longrightarrow \mathcal{D}$  two  dg-functors. A \textbf{dg-natural transformation of degree n}  denoted by  $\eta: F \longrightarrow G$,  is a family of morphisms $\{ \eta_{X}: F(X) \longrightarrow G(X)\}_{X \in \mathcal{C}}$ such that $\eta_{X} \in \Ho_{\mathcal{D}} ^{n} (F(X), G(X))$ for each $X \in \mathcal{C}$ and for all $f \in \Ho_{\mathcal{C}} ^{m} (X, Y)$ homogeneous morphism of degree $m$,  the following diagram  commutes up to the sign $(-1)^{nm}$ 
$$\xymatrix{ F(X) \ar[r]^{\eta_{X}} \ar[d]_{F(f)}& G(X) \ar[d]^{G(f)} \\
  F(Y) \ar[r]_{\eta_{Y}} & G(Y), }$$
i.e.,  $G(f)\circ \eta_{X} = (-1)^{nm} \eta_{Y}\circ F(f)$. Let us denote by $\mathrm{DgNat}^{n}(F,G)$ the set of all the dg-natural transformation of degree $n$ and we define:
$$\mathrm{DgNat}(F,G):=\bigoplus_{n\in \mathbb{Z}}\mathrm{DgNat}^{n}(F,G),$$
as the set of all the $\textbf{dg}$-$\textbf{natural transformations}$ from $F$ to $G$.
\end{definition} 

Let $\mathcal{C}$ and $\mathcal{D}$ be  small dg-categories. Taking into account the dg-natural transformations one can now construct a new category whose objects will be dg-functors going from $\C$ to $\D$ and whose morphisms are the natural dg-transformations $\mathrm{DgNat}(F,G)$. In this way,  the $\textbf{category of dg-covariant functors}$ denoted by $\mathrm{DgFun}(\C,\D)$ is defined as follows. 
\begin{enumerate}
\item [(a)] The class of objects are the dg-covariant functors $F: \C \longrightarrow \D$.
\item [(b)] For $F,G  \in \mathrm{DgFun}(\C,\D)$ the set of morphisms will be by definition the set of natural dg-transformations from $F$ to $G$, i.e,
$$\Ho_{\mathrm{DgFun}(\C,\D)}(F,G):= \mathrm{DgNat}(F,G).$$

\end{enumerate}
Moreover each space of dg-natural transformations $\mathrm{DgNat}(F,G)$ has a structure of differential graded $K$-module given by the differential $d:\mathrm{DgNat}(F,G)\longrightarrow \mathrm{DgNat}(F,G)$ which is a graded $K$-linear morphism of degree $1$ such that $d \circ d = 0$ and is defined by the rule:

\begin{equation}\label{diferentialNat}
d(\eta)_{X}:=d_{\mathrm{Hom}_{\mathcal{D}}(F(X),G(X))}(\eta_{X}),\,\,\, \forall \eta\in \mathrm{DgNat}^{n}(F,G)
\end{equation}
and for all $X\in \mathcal{C}$, where the morphism $d_{\mathrm{Hom}_{\mathcal{D}}(F(X),G(X))}:\mathrm{Hom}_{\mathcal{D}}(F(X),G(X))\longrightarrow \mathrm{Hom}_{\mathcal{D}}(F(X),G(X))$ is the differential in the set $\mathrm{Hom}_{\mathcal{D}}(F(X),G(X))$ coming from the fact that $\mathcal{D}$ is a dg-category.

The following proposition is easy to verify.

\begin{proposition}\label{Hom(C,D) dg cat}
Let $\mathcal{C}$, $\mathcal{D}$ be  small dg-categories. Then the category $\mathrm{DgFun}(\C,\D)$ is a dg-category.
 \end{proposition}

If we take $K$ as the dg-category with the single object $\{\star\}$ as the unit
object, we have that the category $\mathrm{dg}$-$\mathrm{Cat}_{K}$ is a symmetric tensor category, i.e. we have the
adjunction
$$ \Ho_{\mathrm{dg}-\mathrm{Cat}_{K}}(\A \otimes \B, \C)\cong \Ho_{\mathrm{dg}-\mathrm{Cat}_{K}}(\A,\mathrm{Hom}_{\mathrm{dg}-\mathrm{Cat}_{K}}(\B,\C))$$
for $\A,\B,\C$ dg-categories. The pair $(\otimes,\Ho)$ makes $\mathrm{dg}$-$\mathrm{Cat}_{K}$ into a closed symmetric monoidal category (see \cite{Kelly}). This structure will be important for the remainder of this work  since the immediate consequence of the above fact is that there exists an isomorphism of dg-categories

\begin{equation}\label{bifuntor2variables}
\mathrm{DgFun}(\A \otimes \B, \C)\cong \mathrm{DgFun}(\A,\mathrm{DgFun}(\B,\C)).
\end{equation}
Note that this  isomorphism is related to the Lemma \ref{AxB dg funtor}.\\
Just as dg-algebras generalize to dg-categories, the same game can be played with dg-modules and this is crucial for the rest of the paper. We will introduce dg-modules.

\begin{definition}
Let $\C$ be a small dg-category. We  define a  {\textbf{left dg $\C$-module}} to be a dg-functor  $M: \C \longrightarrow \mathrm{DgMod}(K)$ while a right dg $\C$-module is a dg-functor  $N: \C^{op} \longrightarrow \mathrm{DgMod}(K)$.
\end{definition}

Now, we define the dg-category of left dg $\mathcal{C}$-modules.
\begin{definition}\label{DfModC}
Let $\C$ be a small dg-category. The category of left dg $\C$-modules denoted by $\mathrm{DgMod}(\C)$ has all dg $\C$-modules as objects and morphisms of dg-functors as the dg-natural transformations. That is,
$$\mathrm{DgMod}(\C):=\mathrm{DgFun}(\mathcal{C},\mathrm{DgMod}(K)).$$
\end{definition}

By Proposition \ref{Hom(C,D) dg cat}, we have that $\mathrm{DgMod}(\C)$ is a dg-category such that for any dg-functors $F,G: \C \longrightarrow \mathrm{DgMod}(K)$

\begin{equation}\label{morfismosdgmod}
\Ho_{\mathrm{DgMod}(\C)}(F,G)= \bigoplus_{n \in \Z} \Ho^{n}_{\mathrm{DgMod}(\C)}(F,G),
\end{equation}

where by definition we set $\Ho^{n}_{\mathrm{DgMod}(\C)}(F,G):=\mathrm{DgNat}^{n}(F,G)$ (see Definition \ref{dgnaturaltrans}). Moreover,  $\Ho_{\mathrm{DgMod}(\C)}(F,G)$ has a structure of differential graded $K$-module given by the differential $d: \Ho_{\mathrm{DgMod}(\C)}(F,G)\longrightarrow  \Ho_{\mathrm{DgMod}(\C)}(F,G)$ which is a graded $K$-linear morphism of degree $1$ such that $d \circ d = 0$ and is defined by the rule:

\begin{equation}\label{dgNatmod}
d(\eta)_{X}:=d_{\mathrm{Hom}_{\mathrm{DgMod}(K)}(F(X),G(X))}(\eta_{X}),\,\,\, \forall \eta\in  \Ho^{n}_{\mathrm{DgMod}(\C)}(F,G)
\end{equation}
and for all $X\in \mathcal{C}$, where the morphism 
$$d_{\mathrm{Hom}_{\mathrm{DgMod}(K)}(F(X),G(X))}\!:\!\mathrm{Hom}_{\mathrm{DgMod}(K)}(F(X),G(X))\rightarrow \mathrm{Hom}_{\mathrm{DgMod}(K)}(F(X),G(X))$$ is the differential  defined in the Equation  \ref{Homdgestructure}.\\

Now we recall the Yoneda's Lemma and Yoneda's emdedding in the context of differential graded categories, for more details see Section 4.3 in  \cite{Takeda} or Equation (2.34) in p. 34 in  \cite{Kelly} for the general context of enriched categories.

\begin{proposition}(Dg  Yoneda's Lemma)
Let $M\in \mathrm{DgMod}(\mathcal{C})$ be, then there exists an isomorphism of complexes
\begin{align*}
\mathrm{Hom}_{\mathrm{DgMod}(\mathcal{C})}\Big(\mathrm{Hom}_{\mathcal{C}}(C,-),M\Big)\simeq M(C)\\
\eta\mapsto \eta_{C}(1_{C})
\end{align*}
which is natural in $M$ and $C$.
\end{proposition}

\begin{proposition}(Dg-Yoneda's embedding)\label{dgYoneda}
Let $\mathcal{C}$ be a dg-category. Then the dg-Yoneda embedding
$$Y:\mathcal{C}\longrightarrow \mathrm{DgMod}(\mathcal{C}^{op})$$
given as $Y(C):=\mathrm{Hom}_{\mathcal{C}}(-,C)$  is a fully faithful dg-funtor.
\end{proposition}

In the following Definition we follow the notation used in \cite[Definition 3.2.2]{Yeku1} on page 72.

\begin{definition}
The category which will be  denoted by $\bf{DgMod}_{\textbf{Str}}(K)$ has as objects the dg $K$-modules; and
given $M$ and $N$ two dg $K$-modules the set of morphisms from $M$ to $N$ in the category $\mathrm{DgMod}_{str}(K)$ is by definition all the morphisms  $f:M\longrightarrow N$ of $K$-modules of degree 0 that commutes with the differentials. That is,
$$\Ho_{\mathrm{DgMod}_{str}(K)}(M, N ):= \{f:M\longrightarrow N\mid f(M^{i})\subseteq N^{i}\,\,\forall i\in \mathbb{Z}\,\,\text{and}  f\circ d_{M}=d_{N}\circ f\}.$$
\end{definition}
We note that we can think $\mathrm{DgMod}_{Str}(K)$ as the category of complexes of $K$-modules.\\
It is well known that $\mathrm{DgMod}_{str}(K)$ is a $\textbf{symmetric monoidal closed category}$ where the unit is the module $K$ and $a_{X,Y,Z}:(X\otimes Y)\otimes Z\longrightarrow X\otimes (Y \otimes Z)$, $l_{X}:K\otimes X:\longrightarrow X$, $r_{X}:X\otimes K\longrightarrow X$ are the natural isomorphisms.\\
In the language of enriched categories (see \cite{Kelly}) we have that a dg-category is just a category enriched over $\mathrm{DgMod}_{str}(K)$.\\

\section{A triple adjoint of dg-functors}

We recall that if $R$ and $S$ are associative rings with units and $\varphi:R\longrightarrow S$ is a morphism of rings we have the functor restriction of scalars $\varphi_{\ast}:\mathrm{Mod}(S)\longrightarrow \mathrm{Mod}(R)$, where $\mathrm{Mod}(S)$ denotes the category of left $S$-modules for a ring with unit $S$. In this case it is well known that $\varphi_{\ast}$ has a left and right adjoint. That is, we have the following diagram

$$\xymatrix{\mathrm{Mod}(S)\ar[rrr]|{\varphi_{\ast}}  &  &&\mathrm{Mod}(R)
\ar@<-2ex>[lll]_{\varphi^{\ast}}\ar@<2ex>[lll]^{\varphi^{!}}}$$
where $(\varphi^{\ast}, \varphi_{\ast})$ and $(\varphi_{\ast}, \varphi^{!})$ are adjoint pairs.
We have the following Definition that is a generalization of restriction of scalars in module theory.

\begin{definition}
Let $F:\mathcal{A}\longrightarrow \mathcal{C}$ a dg-functor. We have the induced dg-functor
$$F_{\ast}:\mathrm{DgMod}(\mathcal{C})\longrightarrow \mathrm{DgMod}(\mathcal{A})$$
given as $F_{\ast}(S):=SF$ for all $S\in \mathrm{DgMod}(\mathcal{C})$.
\end{definition}
The goal of all this section is to show that $F_{\ast}$ has a left and right adjoint (see Proposition \ref{dgadjuninduced}). The left and right adjoints of $F_{\ast}$ are computed as left and right Kan extensions, respectively. So, in the Subsection \ref{endcoend}  we first recall the notions of ends and coends, then  in  the Subsection \ref{weightedlimi} we study weighted limits and colimits in the dg-context.  In  the Subsection \ref{Kanexten} we will give the notion of Kan extensions, and finally in the Subsection \ref{PropertyKan}, we recall the universal property that has the left and right Kan extensions  which help us to see that $F_{\ast}$ has left and right adjoints.

\subsection{Ends and coends in dg-categories}\label{endcoend}

Since a dg-category is just a category enriched over  $\mathrm{DgMod}_{str}(K)$, in this section we will make use of the concepts and results developed in Kelly's book \cite{Kelly}, and some other definitions and results from Emily Riehl's book (\cite{Emily1}) for the particular case of the symmetric monoidal closed category $\mathcal{V}=(\textbf{DgMod}_{\textbf{Str}}(K),\otimes, a,l, r, K)$. First we recall the following notion defined for dg-categories.

\begin{definition}\label{closezeromap}
Let $\mathcal{A}$ be  a $\mathrm{dg}$-category. We say that $\alpha\in \mathcal{A}(X,Y)$ is a $\textbf{closed}$ $\textbf{degree}$ $\textbf{zero}$ $\textbf{map}$ if $\alpha \in \mathcal{A}(X,Y)^{0}$ and $d_{\mathcal{A}(X,Y)}(f)=0.$
\end{definition}

Using that $\mathcal{V}=(\textbf{DgMod}_{\textbf{Str}}(K),\otimes, a,l, r, K)$ is a  symmetric monoidal closed category, the definition of end given in \cite{Kelly} on page 27 can be translated in the following.

\begin{definition} \label{def dg end3}
Let $T:\mathcal{A}^{op} \otimes \mathcal{A} \longrightarrow \mathrm{DgMod}(K)$ be dg-functor. An \textbf{end for T} is an object $\int_{A \in \mathcal{A}} T(A,A)  \in \mathrm{DgMod}(K)$ together with a family of closed degree zero maps 
$$\left\{\lambda_{A}: \int_{A \in  \mathcal{A}}F(A,A) \longrightarrow F(A,A) \right\}_{A \in  \mathcal{A}}$$
in $\mathrm{DgMod}(K)$ such that for  each $f \in \mathcal{A}(A,B)$ the following conditions are satisfied.  
\begin{enumerate}
\item[(a)] The following diagram commutes in $\mathrm{DgMod}(K)$:
$$ \resizebox{0.4 \textwidth}{!}{ \xymatrix{ \int_{A \in \mathcal{A}}T(A,A) \ar[r]^{\lambda_{A}}\ar[d]_{\lambda_{B}}& T(A,A)\ar[d]^{T(A,f)}\\
 T(B,B)\ar[r]_{T(f,B)} & T(A,B),}}$$
 that is, $T(f,B)\circ \lambda_{B}=T(A,f)\circ \lambda_{A}$.
 \item[(b)] For any other object $Y \in \mathrm{DgMod}(K)$ and another family of closed degree zero maps $\{ \alpha_{A}:Y \longrightarrow T(A,A)\}_{A \in \mathcal{A}}$ such that $T(f,B)\circ \alpha_{B}=T(A,f)\circ \alpha_{A}$, there exists a unique closed degree zero map $\gamma: Y \longrightarrow \int_{A \in \mathcal{A}} T(A,A)$ such that the following diagram commutes 
$$\resizebox{0.4\textwidth}{!}{\xymatrix{ Y \ar@{.>}[rd]^{\gamma} \ar@/^/[rrd]^{\alpha_{A}}\ar@/_/[rdd]_{\alpha_{B}}& & \\&  \int_{A \in \mathcal{A}}T(A,A) \ar[r]^{\lambda_{A}}\ar[d]_{\lambda_{B}}& T(A,A)\ar[d]^{T(A,f)}\\
 &  T(B,B)\ar[r]_{T(f,B)}& T(A,B),  }}$$
 i.e., $\alpha_{A}=\lambda_{A}\circ \gamma$  for each $A \in \mathcal{A}$.
\end{enumerate}
\end{definition}

Dually we have the notion of $\textbf{coend}$ of a dg-functor $T:\mathcal{A}^{op} \otimes \mathcal{A} \longrightarrow \mathrm{DgMod}(K)$, we left to the reader to dualize Definition \ref{def dg end3}.\\

\begin{definition}\label{leftrightaction}
Let $T:\mathcal{A}^{op}\otimes \mathcal{B}\longrightarrow \mathrm{DgMod}(K)$ be a dg-functor.
\begin{enumerate}
\item [(a)] $\textbf{Right}$ $\mathcal{A}$-$\textbf{action}$. Let $f \in \mathcal{A}^{op}(A,A')=\mathcal{A}(A',A)$  of degree $|f|$. For $x\in T(A,B'')$ of degree $|x|$ we define
$$ x \cdot f:= (-1)^{|x||f|} T(f \otimes 1_{B''})(x).$$
\item [(b)] $\textbf{Left}$ $\mathcal{B}$-$\textbf{action}$. Let $g\in \mathcal{B}(B,B')$ of degree $|g|$. For $y\in T(A'',B)$ of degree $|y|$  we define
$$ g\cdot y:=T(1_{A''}\otimes g)(y).$$
\end{enumerate}
\end{definition}

There are more general notions of end and coends where the category $\mathrm{DgMod}(K)$ is replaced by an arbitrary dg-category $\mathcal{B}$. However, the existence of ends and coends can not be assured in that generality. But in the case of $\mathrm{DgMod}(K)$, the following proposition tell us that they always exists.

\begin{proposition}\label{descripendModK}
Let $T:\mathcal{A}^{op}\otimes \mathcal{A}\longrightarrow \mathrm{DgMod}(K)$ be a dg-functor. The end of $T$ is isomorphic to the subcomplex $\mathbf{L}$ of $\prod_{A\in \mathcal{A}}T(A,A)$ whose component of degree $i$ is defined as:
$$\mathbf{L}^{i}:=\left\{x=(x_{A})_{A\in \mathcal{A}}\in \Big(\prod_{A\in \mathcal{A}}T(A,A)\Big)^{i}\mid f \cdot x_{A}=(-1)^{ij} x_{B}\cdot f,\,\,\forall f\in \mathcal{A}^{j}(A,B)\right\}.$$
For each $A\in \mathcal{A}$, the morphism $\lambda_{A}: \mathbf{L}\longrightarrow T(A,A)$ is defined as follows: for $x=(x_{A})_{A\in \mathcal{A}}\in \mathbf{L}^{i}$ we set $\lambda_{A}(x):=x_{A}.$
\end{proposition}
\begin{proof}
See Proposition 3.4 in p. 634 in \cite{Genovese1}.
\end{proof}

The following corollary tells us that the dg-natural transformations given in Definition \ref{dgnaturaltrans} can be computed as an end.

\begin{corollary}\label{construcciondgnat}
Let $F,G:\mathcal{A}\longrightarrow \mathcal{B}$ be dg-functors. Consider the dg-funtor 
$$T:=\mathcal{B}(F,G):\mathcal{A}^{op}\otimes \mathcal{A}\longrightarrow \mathrm{DgMod}(K).$$
Hence, $\mathrm{DgNat}(F,G)=\int_{A\in \mathcal{A}}T(A,A),$  where the family of morphisms
$$\left\{\lambda_{A}:\mathrm{DgNat}(F,G)\longrightarrow \mathcal{B}(F(A),G(A))\right\}_{A\in \mathcal{A}}$$
is defined as follows: for $\eta=\{\eta_{A}\}_{A\in \mathcal{A}}\in \mathrm{DgNat}^{n}(F,G)$  the map $\lambda_{A}$ is given as:
$$\lambda_{A}(\eta)=\eta_{A}\in \mathcal{B}(F(A),G(A))^{n}.$$
\end{corollary}

Now, we state the following Proposition dual to Proposition \ref{descripendModK}.

\begin{proposition}\label{coendexistedg}
Let $T:\mathcal{A}^{op}\otimes \mathcal{A}\longrightarrow \mathrm{DgMod}(K)$ be a dg-functor. Then the coend $\int^{A \in \mathcal{A}}T(A,A)$ of  $T$  exists and  
$$\int^{A \in \mathcal{A}}T(A,A)\cong \mathrm{Coker} \Big(\theta:\coprod_{A_{1},A_{2}\in \mathcal{A}} \mathcal{A}(A_{2},A_{1})\otimes T(A_{1},A_{2})\longrightarrow \coprod_{A \in \mathcal{A}} T(A,A) \Big)$$ where $\theta(f \otimes x)=T( 1_{A_{1}} \otimes f)(x)-(-1)^{|f||x|}T(f \otimes 1_{A_{2}})(x)$ for homogeneous elements $f\in \mathcal{A}(A_{2},A_{1})$ and $x\in T(A_{1},A_{2})$. Moreover, for each $A\in \mathcal{A}$, the morphism $\mu_{A}: T(A,A)\longrightarrow \int^{A \in \mathcal{A}}T(A,A)$ is defined as the composition  of the inclusion $i_{T(A,A)}:T(A,A)\longrightarrow \coprod_{A} T(A,A)$ with $p: \coprod_{A\in \mathcal{A}} T(A,A) \longrightarrow \int^{A \in \mathcal{A}}T(A,A)$ where $p$ is the cokernel map of $\theta$.
\end{proposition}
\begin{proof}
See Proposition 3.6 in p. 635 in \cite{Genovese1}
\end{proof}

\subsection{Weighted limits and colimits in dg-categories}\label{weightedlimi}

To compute some Kan extensions in the context of dg-categories, the ordinary limits defined in terms of the representability of cones are not enough. We require a broader notion of limit, known as a weighted limit. In this section we will continue working in the enriched context on $\mathcal{V}$ a symmetric and closed monoidal category, in particular, for $\mathcal{V}=(\textbf{DgMod}_{\textbf{Str}}(K),\otimes, a,l, r, K)$. We follow page 37 in \cite{Kelly}.

\begin{definition}
Let $W:\mathcal{A}\longrightarrow \mathrm{DgMod}(K)$ a dg-functor with $\mathcal{A}$ small we will say that $W$ is an $\textbf{indexing functor}$ or a $\textbf{weight}$. Given another dg-functor $F:\mathcal{A}\longrightarrow \mathcal{B}$ and having in mind the weight $W$ we will say that $F$ is a $\textbf{diagram in}$ $\mathcal{B}$ $\textbf{of type}$ $W$ or $F$ is $\textbf{weighted by}$ $W$.
\end{definition}
Just for having in mind all the information we will consider the diagram
$$\xymatrix{\mathcal{A}\ar[r]^{F}\ar@{-->}[d]_{W} & \mathcal{B}\\
\mathrm{DgMod}(K) & }$$
Recall that the dg-category of all  the covariant dg-functors from $\mathcal{A}$ to $\mathrm{DgMod}(K))$ is denoted by $\mathrm{DgMod}(\mathcal{A})$ (see Definition \ref{DfModC}).
Now, let $B\in \mathcal{B}$ and  consider the dg-functor $\mathcal{B}(B,-):\mathcal{B}\longrightarrow \mathrm{DgMod}(K)$. For $F:\mathcal{A}\longrightarrow \mathcal{B}$ a dg-functor we have the dg-functor
$$\mathcal{B}(B,F(-)):=\mathcal{B}(B,-)\circ F:\mathcal{A}\longrightarrow \mathrm{DgMod}(K).$$
In the following definition we will use the notation given in \cite[Definition 7.1.1]{Emily1} on page 100. 

\begin{definition}\label{weightedlim}
Let $F:\mathcal{A}\longrightarrow \mathcal{B}$ be a dg-functor. A $\textbf{limit of}$ $F:\mathcal{A}\longrightarrow \mathcal{B}$ $\textbf{weighted by}$ $W:\mathcal{A}\longrightarrow  \mathrm{DgMod}(K)$ is an object $\mathrm{lim}^{W}F\in \mathcal{B}$ and a natural isomorphism in  $\mathrm{DgMod}_{Str}(K)$:
$$\mathcal{B}\Big(B,\mathrm{lim}^{W}F\Big)\simeq \mathrm{DgNat}\Big(W,\mathcal{B}(B,F(-)\Big)$$
for all $B\in \mathcal{B}$.
\end{definition}
We note that in page 37 in \cite{Kelly}, the author uses the notation $\{W,F\}$ for $\mathrm{lim}^{W}F$. The following Proposition help us to compute weighted limits in a very special case.

\begin{proposition}\label{limcasofacil}
Let $F, W:\mathcal{A}\longrightarrow \mathrm{DgMod}(K)$ be dg-functors. Hence
$$\mathrm{lim}^{W}F=\{W,F\}=\mathrm{DgNat}(W,F).$$
\end{proposition}
\begin{proof}
For simplicity we use the following notation  for the functor $\mathrm{Hom}$:
$$[-,-]=\mathrm{Hom}_{\mathrm{DgMod}(K)}(-,-):\mathrm{DgMod}(K)^{op}\otimes \mathrm{DgMod}(K)\longrightarrow \mathrm{DgMod}(K).$$
By formula 1.27 on page 14 in Kelly's book \cite{Kelly}, we have an isomorphism:
$$\Big[L\otimes M,N\Big]\simeq \Big[L, [M,N]\Big].$$
We have that $\mathrm{lim}^{W}F\in \mathrm{DgMod}(K)$ satisfies that for each $B\in \mathrm{DgMod}(K)$ the following:
\begin{align*}
[B,\mathrm{lim}^{W}F] & \cong \mathrm{DgNat}\Big(W,[B,F(-)]\Big) & \text{(definition of weighted limit)}\\
& \cong \int_{A \in \mathcal{A}}\Big[W(A),[B,F(A)]\Big]&\text{( by Corollary \ref{construcciondgnat})}\\
& \cong  \int_{A \in \mathcal{A}}\Big[W(A)\otimes B,F(A)\Big]&\text{(adjointnes)}\\
& \cong \int_{A \in \mathcal{A}}\Big[B\otimes W(A),F(A)\Big] &\text{(symmetry)}\\
& \cong \int_{A \in \mathcal{A}}\Big[B, [W(A),F(A)]\Big]  & \text{(adjointnes)}\\ 
& \cong \Big[B, \int_{A \in \mathcal{A}}[W(A),F(A)]\Big] & [B,-]\,\,\, \text{(preserve ends)}  
\end{align*}  
The fact that $[B,-]$ preserve ends, can be checked in formula 2.3 on page 28 in\cite{Kelly}. Hence, by Yoneda's Lemma  and  Corollary \ref{construcciondgnat}, we get that $\mathrm{lim}^{W}F \simeq \int_{A \in \mathcal{A}}[W(A),F(A)]=\mathrm{DgNat}(W,F)$.
\end{proof}

Now, we can discuss the notion of weighted colimit. Let $W:\mathcal{A}^{op}\longrightarrow \mathrm{DgMod}(K)$ a dg-functor with $\mathcal{A}$ small. Consider another dg-functor $F:\mathcal{A}\longrightarrow \mathcal{B}$ and for $B\in \mathcal{B}$ we construct the dg-functor $\mathcal{B}(-,B):\mathcal{B}^{op}\longrightarrow \mathrm{DgMod}(K)$.  Hence, for $F:\mathcal{A}\longrightarrow \mathcal{B}$ a dg-functor we have the dg-functor
$$\mathcal{B}(F(-),B):=\mathcal{B}(-,B)\circ F^{op}:\mathcal{A}^{op}\longrightarrow \mathrm{DgMod}(K).$$

In the following definition we will use the notation given in \cite[Definition 7.2.1]{Emily1} on page 105. 

\begin{definition}\label{weightedcolim}
Let $F:\mathcal{A}\longrightarrow \mathcal{B}$ be a dg-functor. A $\textbf{colimit of}$ $F:\mathcal{A}\longrightarrow \mathcal{B}$ $\textbf{weighted by}$ $W:\mathcal{A}^{op}\longrightarrow  \mathrm{DgMod}(K)$ is an object $\mathrm{colim}^{W}F\in \mathcal{B}$ and an natural isomorphism in  $\mathrm{DgMod}_{Str}(K)$:
$$\mathcal{B}\Big(\mathrm{colim}^{W}F,B\Big)\simeq \mathrm{DgNat}\Big(W,\mathcal{B}(F(-),B)\Big)$$
 for all $B\in \mathcal{B}$ (the right hand side means the dg-natural transformations in the category $\mathrm{DgMod}(\mathcal{A}^{op})=\mathrm{DgFun}\Big(\mathcal{A}^{op},\mathrm{DgMod}(K)\Big)$).
\end{definition}

The following example give us an easy way to compute weighted colimits in special cases.

\begin{example}\label{coyoneda}
Let $F:\mathcal{A}\longrightarrow \mathcal{B}$ a dg-functor and $W:=\mathrm{Hom}_{\mathcal{A}}(-,A):\mathcal{A}\longrightarrow \mathrm{DgMod}(K)$. Hence,
\begin{align*}
\mathcal{B}(\mathrm{colim}^{W}F,B)\simeq \mathrm{DgNat}\Big(W, \mathcal{B}(F(-),B)\Big) & =\mathrm{DgNat}\Big(\mathrm{Hom}_{\mathcal{A}}(-,A), \mathcal{B}(F(-),B)\Big)\\
& \simeq \mathcal{B}(F(A),B)\quad [\text{dg-Yoneda's Lema}]
\end{align*}
Hence by Yoneda's Lemma we have that $\mathrm{colim}^{W}F=\mathrm{colim}^{\mathrm{Hom}_{\mathcal{A}}(-,A)}F\simeq F(A).$
\end{example}

\begin{Remark}\label{prd2func}
Now, let us consider $F:\mathcal{A}^{op}\longrightarrow \mathrm{DgMod}(K)$  and $G:\mathcal{A}\longrightarrow \mathrm{DgMod}(K)$ dg-functors. Since we have the dg-functor
 $-\otimes -:\mathrm{DgMod}(K)\otimes \mathrm{DgMod}(K)\longrightarrow \mathrm{DgMod}(K),$ we can define a dg-functor
 $$T:\mathcal{A}^{op}\otimes \mathcal{A}\longrightarrow \mathrm{DgMod}(K)$$
 as follows: $T(A,A')=F(A)\otimes G(A')$ for all $A,A'\in \mathcal{A}$. By notation we define $F\otimes G$ as this bifunctor, that is,
 $F\otimes G:=T$. By Proposition \ref{coendexistedg}, we can compute the coend of this bifunctor: $\int^{A\in \mathcal{A}}F(A)\otimes G(A).$
 \end{Remark}
 
 Hence, we have the following definition.

 \begin{definition}\label{tensorprodascoend}
 Let $F:\mathcal{A}^{op}\longrightarrow \mathrm{DgMod}(K)$  and $G:\mathcal{A}\longrightarrow \mathrm{DgMod}(K)$ dg-functors.
 We define 
 $$F\otimes_{\mathcal{A}}G:=\int^{A\in \mathcal{A}}F(A)\otimes G(A)\in \mathrm{DgMod}(K).$$
 Note that we have the index $\mathcal{A}$ in the above tensor product in order to make a difference between $F\otimes_{\mathcal{A}}G$ and  the functor  $F\otimes G:=T$ defined in the Remark \ref{prd2func}.
 \end{definition}
 
 The following Proposition tell us how to compute weighted colimits in a special case.
 
 \begin{proposition}\label{tesorcasofacil}
  Let $F:\mathcal{A}^{op}\longrightarrow \mathrm{DgMod}(K)$  and $G:\mathcal{A}\longrightarrow \mathrm{DgMod}(K)$ dg functors. Then
  $\mathrm{colim}^{F}G$ exists and 
  $$\mathrm{colim}^{F}G\simeq F\otimes_{\mathcal{A}}G.$$
 \end{proposition}
 \begin{proof}
 First recall that we have $[-,-]:=\mathrm{Hom}_{ \mathrm{DgMod}(K)}(-,-): \mathrm{DgMod}(K)^{op}\otimes  \mathrm{DgMod}(K)\longrightarrow \mathrm{DgMod}(K).$ For $B\in \mathrm{DgMod}(K)$ we have that
 \begin{align*}
\mathrm{DgNat}\Big(F, [G(-),B]\Big)& \simeq  \int_{A\in \mathcal{A}}[F(A),[G(A),B]] & \text{(by Corollary \ref{construcciondgnat})}\\
 & \simeq  \int_{A\in \mathcal{A}}[F(A)\otimes G(A),B] & \text{(simmetry)}\\
 & \simeq \Big[\int^{A\in \mathcal{A}}F(A)\otimes G(A),B\Big] & ([-,B]\,\, \text{change ends for coends})\\
 & =\Big[F\otimes_{\mathcal{A}}G,B]. & (\text{Definition}\,\, \ref{tensorprodascoend})
 \end{align*}
The fact that $[-,B]$ change ends for coends is dual to the formula 2.3 on page 28 in  \cite{Kelly}. Recall that if $\mathrm{colim}^{F}G\in \mathrm{DgMod}(K)$ exists, it must satisfy that $[\mathrm{colim}^{F}G, B]\simeq \mathrm{DgNat}\Big(F, [G(-),B]\Big)$ for all $B\in \mathcal{B}$. Since $F\otimes_{\mathcal{A}}G$ exists by Proposition \ref{coendexistedg}, we conclude that $\mathrm{colim}^{F}G=F\otimes_{\mathcal{A}}G.$
 \end{proof}
 
\begin{corollary}\label{tensorhom}
For each $A\in \mathcal{A}$ and $G:\mathcal{A}\longrightarrow \mathrm{DgMod}(K)$ we have that
$$\mathrm{Hom}_{\mathcal{A}}(-,A)\otimes_{\mathcal{A}}G\simeq G(A).$$
\end{corollary}
\begin{proof}
It follows from Example \ref{coyoneda} and Proposition \ref{tesorcasofacil}.
\end{proof}

\subsection{Right and Left Kan extensions}\label{Kanexten}

Now, we will define the notions of right and left Kan extensions for dg-categories. Here, we will follow  Chapter 4 in \cite{Kelly}. First, we have the following definition.

\begin{definition}
Let $F,G:\mathcal{A}\longrightarrow \mathcal{B}$ dg functors and consider $\psi\in \mathrm{DgNat}(F,G)$ (see Definition \ref{dgnaturaltrans}). We say that $\psi$ is a $\textbf{closed and degree zero dg-natural}$ $\textbf{transformation}$ if  for each component $\psi_{A}$ of $\psi$ we have $\psi_{A}\in \mathcal{B}(F(A),G(A))^{0}$, $d_{\mathcal{B}(F(A),G(A))}(\psi_{A})=0$ and for all $f\in \mathcal{A}(A,A')$ the following diagram commutes
$$\xymatrix{F(A)\ar[r]^{\psi_{A}}\ar[d]_{F(f)} & G(A)\ar[d]^{G(f)}\\
F(A')\ar[r]_{\psi_{A'}} & G(A').}$$
That is, $\psi$ is a usual natural transformation whose components are  closed degree zero maps (see Definition \ref{closezeromap}).
\end{definition}

In the following discussion , we will use the context given in the following Remark.

\begin{Remark}\label{condicionesKan}
Let $F:\mathcal{A}\longrightarrow \mathcal{C}$, $W:\mathcal{C}\longrightarrow \mathrm{DgMod}(K)$,  $T:\mathcal{C}\longrightarrow \mathcal{B}$ and $G:\mathcal{A}\longrightarrow \mathcal{B}$ dg-functors and we consider the following diagram

\begin{equation}\label{Kan diagram1}
\begin{tikzcd}[ampersand replacement=\&]
	\& {\mathcal{C}} \\
	{\mathcal{A}} \&\& {\mathcal{B}}
	\arrow["F", from=2-1, to=1-2]
	\arrow["T", from=1-2, to=2-3]
	\arrow[""{name=0, anchor=center, inner sep=0}, "G"', from=2-1, to=2-3]
	\arrow["\psi"', shorten <=3pt, shorten >=5pt, Rightarrow, from=1-2, to=0]
\end{tikzcd}
\end{equation}
 where $\psi$ is a closed and degree zero dg-natural transformation in $\mathrm{DgNat}(TF,G)$.
\end{Remark}

Consider dg-functors $F:\mathcal{A} \longrightarrow \mathcal{C}$, $T:\mathcal{C} \longrightarrow \mathcal{B}$, $G:\mathcal{A}\longrightarrow \mathcal{B}$  and dg-natural transformation  $\psi:TF \longrightarrow G$ as in the Remark \ref{condicionesKan}.
For each dg-funtor $W: \mathcal{C} \longrightarrow \mathrm{DgMod}(K)$, there exists the following morphisms
$$ \resizebox{1 \textwidth}{!}{\xymatrix{\mathrm{DgNat}\Big(W, \mathcal{B}(B,T(-))\Big)\ar[r]^{\gamma} &   \mathrm{DgNat}\Big(WF, \mathcal{B}(B,TF(-))\Big)\ar[r]^{\Delta} & \mathrm{DgNat}\Big(WF, \mathcal{B}(B,G(-))\Big)}}$$
Since, $\mathrm{DgNat}\Big(W, \mathcal{B}(B,T(-))\Big)\simeq \mathcal{B}\Big(B,\mathrm{lim}^{W}T\Big)$, and $\mathrm{DgNat}\Big(WF, \mathcal{B}(B,G(-))\Big)\simeq \mathcal{B}\Big(B,\mathrm{lim}^{WF}G\Big)$, we have $\Delta\circ \gamma$ induces a morphism $(F,\psi)_{\ast}:\mathrm{lim}^{W}T\longrightarrow \mathrm{lim}^{WF}{G}.$

 \begin{proposition}$\textnormal{\cite[Theorem 4.6]{Kelly}}$ \label{equivakanext}
 Consider a diagram as in Remark \ref{condicionesKan}. The following conditions are equivalent.
\begin{enumerate}
\item [(a)] For each $W:\mathcal{C}\longrightarrow \mathrm{DgMod}(K)$ the morphism
$$(F,\psi)_{\ast}:\mathrm{lim}^{W}T\longrightarrow \mathrm{lim}^{WF}{G}$$
is an isomorphism,  either limit existing if the other does.

\item [(b)] For each $C\in \mathcal{C}$ we have that $\mathrm{lim}^{\mathcal{C}(C,F(-))}G=T(C)$ with counit given by
$$(\ast): \xymatrix{\mathcal{C}(C,F(-))\ar[d]^{T}\ar[drr]^{\mu_{C}} &\\
 \mathcal{B}(T(C),TF(-))\ar[rr]^{\mathcal{B}(1,\psi_{-})}  & & \mathcal{B}(T(C),G(-)).}$$
That is, for each $A\in \mathcal{A}$ we have that
$$[\mu_{C}]_{A}:\mathcal{C}(C,F(A))\longrightarrow  \mathcal{B}(T(C),G(A))$$ is defined as follows: for $f:C\longrightarrow F(A)$ we set
$$[\mu_{C}]_{A}(f):= \mathcal{B}(1,\psi_{-})_{A}(T(f))=\psi_{A}\circ T(f)\in  \mathcal{B}(T(C),G(A)).$$

\item [(c)] For each $B\in \mathcal{B}$ and $C\in \mathcal{C}$ the map induced by $(\ast)$ in item (b), 
$$\Gamma_{B}:\mathcal{B}(B,T(C))\longrightarrow \mathrm{DgNat}\Big(\mathcal{C}(C,F(-)),\mathcal{B}(B, G(-))\Big)$$
is an isomorphism. We note that for $\beta:B\longrightarrow T(C)$ we have $\mathcal{B}(\beta,G(-)):\mathcal{B}(T(C),G(-))\longrightarrow \mathcal{B}(B,G(-))$. Hence, 
$$\Gamma_{B}(\beta)=\mathcal{B}(\beta,G(-))\circ \mu_{C}.$$
\end{enumerate}
 \end{proposition}
 
The following Definition can be found on page 60 in \cite{Kelly}.

\begin{definition}
Consider the diagram as given in Remark \ref{condicionesKan}.
If the equivalent conditions of the Proposition \ref{equivakanext} hold, we say that the diagram above exhibits $T$ as the $\textbf{right Kan extension of G along F}$; and we write $T=\mathrm{Ran}_{F}G$, calling $\psi$ the $\textbf{counit}$ of this right Kan extension.
 \end{definition}

\begin{Remark}
The item $(b)$ in Proposition \ref{equivakanext} give us that:
$$(\mathrm{Ran}_{F}G)(C)=\mathrm{lim}^{\mathcal{C}(C,F(-))}G.$$
and $\mathrm{Ran}_{F}G$ exists, by definition, when the limit on the right exists for each $C\in \mathcal{C}$.
\end{Remark}

Just for completeness we recall the dual notion.
\begin{Remark}\label{condicionesKanleft}
Consider $F:\mathcal{A}\longrightarrow \mathcal{C}$,  $S:\mathcal{C}\longrightarrow \mathcal{B}$ and $G:\mathcal{A}\longrightarrow \mathcal{B}$ dg-functors

\begin{equation}
\begin{tikzcd}[ampersand replacement=\&]
	\& {\mathcal{C}} \\
	{\mathcal{A}} \&\& {\mathcal{B}}
	\arrow["F", from=2-1, to=1-2]
	\arrow["S", from=1-2, to=2-3]
	\arrow[""{name=0, anchor=center, inner sep=0}, "G"', from=2-1, to=2-3]
	\arrow["\varphi"', shorten <=5pt, shorten >=5pt, Rightarrow, from=0, to=1-2]
\end{tikzcd}
\end{equation}
where $\varphi:G\longrightarrow  SF$  is a closed and degree zero dg-natural transformation in $\mathrm{DgNat}(G, SF)$.
\end{Remark}

Given dg-funtores $F:\mathcal{A} \longrightarrow \mathcal{C}$, $S:\mathcal{C} \longrightarrow \mathcal{B}$, $G:\mathcal{A}\longrightarrow \mathcal{B}$  and dg-natural transformation  $\varphi:G \longrightarrow SF$ as in the Remark \ref{condicionesKanleft}. For each dg-funtor $H: \mathcal{C}^{op} \longrightarrow \mathrm{DgMod}(K)$  we have the morphisms
$$ \resizebox{1 \textwidth}{!}{\xymatrix{\mathrm{DgNat}\Big(H, \mathcal{B}(S(-),B)\Big)\ar[r] & \mathrm{DgNat}\Big(HF^{op}, \mathcal{B}(SF(-),B)\Big)\ar[r]  & \mathrm{DgNat}\Big(HF^{op}, \mathcal{B}(G(-),B)\Big)}}$$
Since $\mathrm{DgNat}\Big(H, \mathcal{B}(S(-),B)\Big)\simeq  \mathcal{B}\Big(\mathrm{colim}^{H}S,B\Big)$ and $ \mathrm{DgNat}\Big(HF^{op}, \mathcal{B}(G(-),B)\Big)\simeq \mathcal{B}\Big(\mathrm{colim}^{HF^{op}}G, B\Big)$, the above morphism induces a morphism (by Yoneda's Lemma):
 $$(F,\varphi)^{\ast}:\mathrm{colim}^{HF^{op}}G\longrightarrow \mathrm{colim}^{H}S.$$
 
\begin{proposition} $\textnormal{\cite[Dual of Theorem 4.6]{Kelly}}$ \label{equivakanextleft}
Consider a diagram as in Remark \ref{condicionesKanleft}. The following conditions are equivalent.
\begin{enumerate}
\item [(a)] For each $H:\mathcal{C}^{op}\longrightarrow \mathrm{DgMod}(K)$ the morphism $(F,\varphi)^{\ast}:\mathrm{colim}^{HF^{op}}G\longrightarrow \mathrm{colim}^{H}{S}$ is an isomorphism,  either colimit existing if the other does.

\item [(b)] For each $C\in \mathcal{C}$ we have that $\mathrm{colim}^{\mathcal{C}(F(-),C)}G=S(C)$ with unit given by
$$(\ast): \xymatrix{\mathcal{C}(F(-),C)\ar[d]^{S}\ar[drr]^{\nu_{C}} &\\
 \mathcal{B}(SF(-),S(C))\ar[rr]^{\mathcal{B}(\varphi_{-},1)}  & & \mathcal{B}(G(-),S(C)).}$$
That is for each $A\in \mathcal{A}$ the unit
$$[\nu_{C}]_{A}:\mathcal{C}(F(A),C)\longrightarrow  \mathcal{B}(G(A),S(C))$$ is defined as follows: for $f:F(A)\longrightarrow C$ we set
$$[\nu_{C}]_{A}(f):= \mathcal{B}(\varphi_{-},1)_{A}(S(f))=S(f)\circ \varphi_{A}\in  \mathcal{B}(G(A),S(C))$$

\item [(c)] For each $B\in \mathcal{B}$ and $C\in \mathcal{C}$ the map induced by $(\ast)$ in item (b), 
$$\Theta_{B}:\mathcal{B}(S(C),B)\longrightarrow \mathrm{DgNat}\Big(\mathcal{C}(F(-),C),\mathcal{B}(G(-),B)\Big)$$
is an isomorphism. We note that for $\beta:S(C)\longrightarrow B$ we have $\mathcal{B}(G(-),\beta):\mathcal{B}(G(-),S(C))\longrightarrow \mathcal{B}(G(-),B)$. Hence, 
$\Theta_{B}(\beta)=\mathcal{B}(G(-),\beta)\circ \nu_{C}.$
\end{enumerate}
 \end{proposition}
 
\begin{definition}\label{definitionleftkan}
Consider the following diagram as in Remark \ref{condicionesKanleft}.
If the equivalent conditions of the Proposition \ref{equivakanextleft} hold, we say that the diagram above exhibits $S$ as the $\textbf{left Kan extension of G along F}$; and we write $S=\mathrm{Lan}_{F}G$, calling $\varphi$ the $\textbf{unit}$ of this left Kan extension.
\end{definition}

 \subsection{Universal Properties of Kan extensions}\label{PropertyKan}

The left and right Kan extension of $G$ along $F$ has an important universal property. In the following we will state the results for left Kan extensions and left the version for right extensions to the reader.

\begin{proposition}\label{lanadjoint}
Let $F:\mathcal{A}\longrightarrow \mathcal{C}$ and $G:\mathcal{A}\longrightarrow \mathrm{DgMod}(K)$ dg-functors. For any dg functor $S:\mathcal{C}\longrightarrow \mathrm{DgMod}(K)$ we have an isomorphism in  $\mathrm{DgMod}_{Str}(K)$, which is natural in the variable $S$:
\begin{equation}\label{propunivlefkan}
\Delta:\mathrm{DgNat}\Big(\mathrm{Lan}_{F}G, S\Big)\longrightarrow \mathrm{DgNat}\Big(G,SF\Big),
\end{equation}
the unit of this representation being the unit $\varphi$. Dually,
\begin{equation}
\mathrm{DgNat}\Big(S,\mathrm{Ran}_{F}G\Big)\simeq \mathrm{DgNat}\Big(SF,G\Big).
\end{equation}
\end{proposition}
\begin{proof}
We have the following isomorphisms
\begin{align*}
& \mathrm{DgNat}\Big(\mathrm{Lan}_{F}G, S\Big)=\\
 & =\int_{C\in \mathcal{C}}\mathcal{B}\Big((\mathrm{Lan}_{F}G)(C), S(C)\Big) & [\text{Corollary}\,\,\ref{construcciondgnat}]\\
& = \int_{C\in \mathcal{C}}\mathcal{B}\Big(\mathrm{colim}^{\mathcal{C}(F,C)}G, S(C)\Big) & [\text{Proposition}\,\, \ref{equivakanextleft}]\\
& =\int_{C\in \mathcal{C}}\mathrm{DgNat}\Big(\mathcal{C}(F,C),\mathcal{B}(G,S(C))\Big)
& [\text{Definition}\,\, \ref{weightedcolim}]\\
 & =\int_{C\in \mathcal{C}}\int_{A\in \mathcal{A}}\Big[\mathcal{C}(F(A),C),\mathcal{B}(G(A),S(C))\Big] & [\text{Corollary}\,\,\ref{construcciondgnat}]\\
& =\int_{A\in \mathcal{A}}\int_{C\in \mathcal{C}}\Big[\mathcal{C}(F(A),C),\mathcal{B}(G(A),S(C))\Big] & [\text{Fubini}]\\
 & =\int_{A\in \mathcal{A}}\mathrm{DgNat}\Big(\mathcal{C}(F(A),-),\mathcal{B}(G(A),S)\Big) &  [\text{Corollary}\,\,\ref{construcciondgnat}]\\
& =\int_{A\in \mathcal{A}}\mathcal{B}(G(A),S(F(A)) & [\text{Yoneda's Lemma}]\\
& = \mathrm{DgNat}\Big(G,SF\Big), & [\text{Proposition}\,\,\ref{construcciondgnat}]
\end{align*}
where the Fubini's rule can be seen in Equation 2.8 on page 29 in \cite{Kelly}.
\end{proof}

\begin{Remark}\label{descripdeltaespe}
It can be seen that the isomorphism $\Delta:\mathrm{DgNat}\Big(\mathrm{Lan}_{F}G, S\Big)\longrightarrow \mathrm{DgNat}\Big(G,SF\Big)$ given in Proposition \ref{lanadjoint} is defined as: 
$$\Delta(\alpha)= (\alpha\circ F)\circ \varphi, \quad \forall \alpha\in \mathrm{DgNat}\Big(\mathrm{Lan}_{F}G,S\Big).$$
Since $\varphi$ is of degree zero, we have that the pair $(\mathrm{Lan}_{F}G,\varphi)$ satisfies the following $\textbf{universal property}$: for each dg functor $S:\mathcal{C}\longrightarrow \mathrm{DgMod}(K)$ and a $\eta \in \mathrm{DgNat}^{n}(G,  SF)$ of degree $n$, there exists a unique $\alpha\in \mathrm{DgNat}^{n}(\mathrm{Lan}_{F}G, S)$ of degree $n$ such that
$$\Delta(\alpha)= (\alpha\circ F)\circ \varphi=\eta.$$
This implies that for each $\eta\in \mathrm{DgNat}(G,SF)$, there exists a unique $\alpha\in \mathrm{DgNat}(\mathrm{Lan}_{F}G, S)$ such that $\Delta(\alpha)= (\alpha\circ F)\circ \varphi=\eta.$ The previous universal property is depicted  in the following diagram:
$$
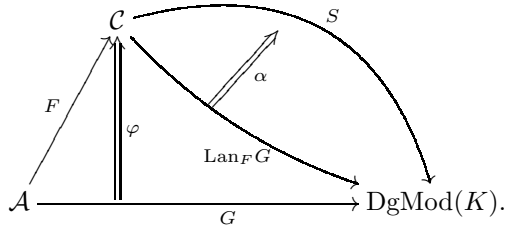
\xymatrix{ &  \mathcal{C}\ar@/_1pc/[ddrrr]_{\mathrm{Lan}_{F}G}\ar@/^3pc/[ddrrr]^{S} & & & & \\
 & & {}\ar@{=>}[ur]_{\alpha} & & &\\
 \mathcal{A}\ar[rrrr]_{G}\ar[uur]^{F}  & {}\ar@{=>}[uu]_(.4){\varphi}\ar@{}[uurrr]|(.5){} & & & \mathrm{DgMod}(K).}$$
\end{Remark}

\begin{definition}\label{dgadjuncion}
\begin{enumerate}
\item [(a)] 
Let $F:\mathcal{A}\longrightarrow \mathcal{B}$ and  $G:\mathcal{A}\longrightarrow \mathcal{A}$ dg-functors. We say that $F$ is $\textbf{dg-left adjoint}$ to $G$ if there exists an isomorphisms in $\mathrm{DgMod}_{str}(K)$
$$\Phi_{A,B}:\mathcal{B}(F(A),B)\longrightarrow \mathcal{A}(A,G(B))$$
natural in $A$ and $B$. In this case we say that $(F,G)$ is a $\textbf{dg-adjoint pair}$.

\item [(b)] Let $F:\mathcal{A}\longrightarrow \mathcal{B}$ be a dg-functor. We say that $F$ is $\textbf{dg-fully}$ $\textbf{faithful}$ if for each $A,A'\in \mathcal{A}$ we have an isomorphism in $\mathrm{DgMod}_{str}(K)$:
$$\mathcal{A}(A,A')\longrightarrow \mathcal{B}(F(A),F(A')).$$
\end{enumerate}
\end{definition}

\begin{Remark}
Let $F:\mathcal{A}\longrightarrow \mathcal{B}$ and  $G:\mathcal{A}\longrightarrow \mathcal{A}$ dg-functors such that  $(F,G)$ is a dg-adjoint pair. In this case the unit $\eta:1_{\mathcal{A}}\longrightarrow GF$ and counit $\varepsilon: FG\longrightarrow 1_{\mathcal{B}}$ are both closed and degree zero natural transformations.
\end{Remark}

 \begin{proposition}\label{existeLan}
Consider a dg-functor $F:\mathcal{A}\longrightarrow \mathcal{C}$. Then there exists a dg-funtor
 $$\mathbb{L}\mathrm{an}_{F}(-): \mathrm{DgMod}(\mathcal{A})\longrightarrow  \mathrm{DgMod}(\mathcal{C})$$
 defined as $\mathbb{L}\mathrm{an}_{F}(G):=\mathrm{Lan}_{F}(G)$.
 \end{proposition}
\begin{proof}
Consider a dg-funtor $F:\mathcal{A}\longrightarrow \mathcal{C}$ and $G\in \mathrm{DgMod}(\mathcal{A})$. We have the following diagram for all $C\in \mathcal{C}$
  $$\xymatrix{ & &\mathrm{DgMod}(K) & \\
 &\mathcal{C}\ar[ur]^(.4){\mathcal{C}(-,C)} & \\
 \mathcal{A}\ar[rr]_{G}\ar[ur]^{F} & {} & \mathrm{DgMod}(K). }$$
By Propositions \ref{equivakanextleft} and \ref{tesorcasofacil} we have that
$$\mathrm{Lan}_{F}(G)(C)=\mathrm{colim}^{\mathcal{C}(F,C)}G=\mathcal{C}(F,C)\otimes_{\mathcal{A}}G.$$
Hence, there exists $\mathrm{Lan}_{F}G:\mathcal{C}\longrightarrow \mathrm{DgMod}(K)$. We define a dg-functor
 $$\mathbb{L}\mathrm{an}_{F}(-): \mathrm{DgMod}(\mathcal{A})\longrightarrow  \mathrm{DgMod}(\mathcal{C})$$
as $\mathbb{L}\mathrm{an}_{F}(G):=\mathrm{Lan}_{F}(G)$ for $G\in \mathrm{DgMod}(\mathcal{A})$. That is, for $G\in \mathrm{DgMod}(\mathcal{A})$ and $C\in \mathcal{C}$:

\begin{equation}
 \mathbb{L}\mathrm{an}_{F}(G)(C):=\mathcal{C}(F,C)\otimes_{\mathcal{A}}G.
\end{equation}
Let us see that $\mathbb{L}\mathrm{an}_{F}(-): \mathrm{DgMod}(\mathcal{A})\longrightarrow  \mathrm{DgMod}(\mathcal{C})$ is a functor. Let
$\eta:G'\longrightarrow G$ a dg-natural transformation in  $\mathrm{DgMod}(\mathcal{A})$. Consider $\mathrm{Lan}_{F}G$, by  Definition \ref{definitionleftkan} and Remark \ref{condicionesKanleft}, there exists a natural transformation $\varphi:G\longrightarrow \mathrm{Lan}_{F}G\circ F$ of degree zero and such that $d(\varphi)=0$. Hence, we get a dg-natural transformation 
$\varphi\eta:G'\longrightarrow  \mathrm{Lan}_{F}G\circ F$. By Proposition \ref{lanadjoint}, we have the isomorphism:
$$\mathrm{Hom}_{\mathrm{DgMod}(\mathcal{C})}\Big(\mathrm{Lan}_{F}G', \mathrm{Lan}_{F}G\Big)\simeq \mathrm{Hom}_{\mathrm{DgMod}(\mathcal{A})}\Big(G',\mathrm{Lan}_{F}G\circ F\Big).$$
Thus, there exists a unique dg-natural transformation $\mathbb{L}\mathrm{an}_{F}(\eta)$ as in the following diagram:
$$\xymatrix{ &  \mathcal{C}\ar@/_1pc/[ddrrr]_{\mathrm{Lan}_{F}G'}\ar@/^3pc/[ddrrr]^{\mathrm{Lan}_{F}G} & & & & \\
 & & {}\ar@{=>}[ur]_{\mathbb{L}\mathrm{an}_{F}(\eta)} & & &\\
 \mathcal{A}\ar[rrrr]_{G}\ar[uur]^{F}\ar@/_3pc/[rrrr]_{G'}  & {}\ar@{=>}[uur]_(.4){\varphi}\ar@{}[uurrr]|(.5){} & & & \mathrm{DgMod}(K)\\
 & &  {}\ar@{=>}[u]^{\eta} & }$$
 such that $\mathbb{L}\mathrm{an}_{F}(\eta)\circ \varphi'=\varphi\circ \eta$, where $\varphi':G'\longrightarrow \mathrm{Lan}_{F}G'\circ F$ is the natural transformation of degree zero given by the fact that $\mathrm{Lan}_{F}G'$ is a left Kan extension (see Definition \ref{definitionleftkan}). Hence, we define 
 $$\mathbb{L}\mathrm{an}_{F}(\eta):\mathbb{L}\mathrm{an}_{F}(G')\longrightarrow \mathbb{L}\mathrm{an}_{F}(G),$$
 as the unique morphism such that   $\mathbb{L}\mathrm{an}_{F}(\eta)\circ \varphi'=\varphi\circ \eta$.\\
 The fact that this defines a functor comes from the uniqueness of $\mathbb{L}\mathrm{an}_{F}(\eta)$. Now, let us see that $\mathbb{L}\mathrm{an}_{F}(-)$ commutes with differentials. That is, let us see that the following diagram commutes for homogenous elements ($d$ denotes different differentials)

$$\xymatrix{\mathrm{Hom}_{\mathrm{DgMod}(\mathcal{A})}(G',G)\ar[rrr]^{\mathbb{L}\mathrm{an}_{F}}\ar[d]_{d} & & & 
 \mathrm{Hom}_{\mathrm{DgMod}(\mathcal{C})}(\mathrm{Lan}_{F}G',\mathrm{Lan}_{F}G)\ar[d]^{d}\\
 \mathrm{Hom}_{\mathrm{DgMod}(\mathcal{A})}(G',G)\ar[rrr]_{\mathbb{L}\mathrm{an}_{F}} & & & 
 \mathrm{Hom}_{\mathrm{DgMod}(\mathcal{C})}(\mathbb{L}\mathrm{an}_{F}G',\mathrm{Lan}_{F}G)}$$
 
 Indeed, let $\eta:G'\longrightarrow G$ of degree $n$.  We have that $\mathbb{L}\mathrm{an}_{F}(\eta)\circ \varphi'=\varphi\circ \eta$. 
On one hand, since $d(\varphi')=0$ we have that:
$$d(\mathbb{L}\mathrm{an}_{F}(\eta)\circ \varphi')=d(\mathbb{L}\mathrm{an}_{F}(\eta))\circ \varphi'+(-1)^{|\mathbb{L}\mathrm{an}_{F}(\eta)|}\mathbb{L}\mathrm{an}_{F}(\eta)\circ d(\varphi')=d(\mathbb{L}\mathrm{an}_{F}(\eta)\circ \varphi'$$
On the other hand, since $|\varphi|=0$ we have that
$$d(\varphi\circ \eta)=d(\varphi)\circ \eta+(-1)^{|\varphi|}\varphi\circ d(\eta)=\varphi\circ d(\eta).$$
Thus, we have that
$$d(\mathbb{L}\mathrm{an}_{F}(\eta))\circ \varphi'=\varphi\circ d(\eta)$$
Since $\mathbb{L}\mathrm{an}_{F}(d(\eta))$ is the unique such that
$$\mathbb{L}\mathrm{an}_{F}((d(\eta))\circ \varphi'=\varphi\circ d(\eta),$$
we conclude that
$$\mathbb{L}\mathrm{an}_{F}((d(\eta))=d(\mathbb{L}\mathrm{an}_{F}(\eta))$$
This proves that the previous diagram commutes. Hence $\mathbb{L}\mathrm{an}_{F}(-)$ conmutes with differential and hence $\mathbb{L}\mathrm{an}_{F}(-): \mathrm{DgMod}(\mathcal{A})\longrightarrow  \mathrm{DgMod}(\mathcal{C})$ is a dg-functor.
\end{proof}
 
Now, we have the dual of Proposition \ref{existeLan}.
 
\begin{proposition}\label{existeRan}
Consider a dg-functor $F:\mathcal{A}\longrightarrow \mathcal{C}$. Then there exists a dg-functor
 $$\mathbb{R}\mathrm{an}_{F}(-): \mathrm{DgMod}(\mathcal{A})\longrightarrow  \mathrm{DgMod}(\mathcal{C})$$
 defined as $\mathbb{R}\mathrm{an}_{F}(G):=\mathrm{Ran}_{F}(G)$.
 \end{proposition}
\begin{proof}
Consider a dg-funtor $F:\mathcal{A}\longrightarrow \mathcal{C}$ and $G\in \mathrm{DgMod}(\mathcal{A})$. We have the following diagram for all $C\in \mathcal{C}$
  $$\xymatrix{ & &  \mathrm{DgMod}(K) & \\
 &\mathcal{C}\ar[ur]^(.4){\mathcal{C}(C,-)} & \\
 \mathcal{A}\ar[rr]_{G}\ar[ur]^{F} & {} & \mathrm{DgMod}(K). }$$
By Propositions \ref{equivakanext} and \ref{limcasofacil} we have that 
 $$\mathrm{Ran}_{F}(G)(C)=\mathrm{lim}^{\mathcal{C}(C,F(-))}G=\mathrm{DgNat}\Big(\mathcal{C}(C,F(-)),G\Big).$$
Hence, there exists $\mathrm{Ran}_{F}G:\mathcal{C}\longrightarrow \mathrm{DgMod}(K)$. We define a dg-functor
 $$\mathbb{R}\mathrm{an}_{F}(-): \mathrm{DgMod}(\mathcal{A})\longrightarrow  \mathrm{DgMod}(\mathcal{C})$$
 defined as $\mathbb{R}\mathrm{an}_{F}(G):=\mathrm{Ran}_{F}(G)$. That is, for $G\in \mathrm{DgMod}(\mathcal{A})$ and $C\in \mathcal{C}$:
 $$\mathbb{R}\mathrm{an}_{F}(G)(C):=\mathrm{DgNat}\Big(\mathcal{C}(C,F(-)),G\Big).$$
As in Proposition \ref{existeLan}, we can check that $\mathbb{R}\mathrm{an}_{F}(-): \mathrm{DgMod}(\mathcal{A})\longrightarrow  \mathrm{DgMod}(\mathcal{C})$
is a dg-functor.
 \end{proof}

\begin{proposition}\label{dgadjuninduced}
Let $F:\mathcal{A}\longrightarrow \mathcal{C}$ a dg-functor and consider the induced functor $F_{\ast}:\mathrm{DgMod}(\mathcal{C})\longrightarrow \mathrm{DgMod}(\mathcal{A})$. Hence the dg-functors  given in Propositions \ref{existeRan} and \ref{existeLan}:

 $$\mathbb{R}\mathrm{an}_{F}(-), \mathbb{L}\mathrm{an}_{F}(-): \mathrm{DgMod}(\mathcal{A})\longrightarrow  \mathrm{DgMod}(\mathcal{C})$$
are the right and left dg-adjoints of $F_{\ast}$.\\
That is, we have the diagram of dg-adjoint pairs:
$$\xymatrix{\mathrm{DgMod}(\mathcal{C})\ar[rrr]|{F_{\ast}}  &  &&\mathrm{DgMod}(\mathcal{A})
\ar@<-2ex>[lll]_{\mathbb{L}\mathrm{an}_{F}(-)}\ar@<2ex>[lll]^{\mathbb{R}\mathrm{an}_{F}(-)}}$$
\end{proposition}
\begin{proof}
It follows by Definition \ref{dgadjuncion} and Propositions   \ref{existeLan},  \ref{existeRan} and  \ref{lanadjoint}.
\end{proof}

\section{Recollement induced by a dg-ideal}
 
\begin{definition}
Let $\mathcal{C}$ be a dg-category. A $\textbf{differential graded ideal}$ $\mathbf{ \mathcal{I}}$ of $\mathcal{C}$ is a  differential graded subfunctor of $\mathrm{Hom}_{\mathcal{C}}(-,-)$. That is, $\mathcal{I}$ is a class of  morphism in $\mathcal{C}$ such that:
\begin{enumerate}
\item[(a)] $\mathcal{I}(A,B)= \mathrm{Hom}_{\mathcal{C}}(A,B) \cap \mathcal{I}$ is a dg-$K$-submodule of $\mathrm{Hom}_{\mathcal{C}}(A,B)$ for each $A, B \in \mathcal{C}$;
\item[(b)] If $f \in \mathcal{I}(A,B)$, $g \in \mathrm{Hom}_{\mathcal{C}}(C,A)$ and $h \in \mathrm{Hom}_{\mathcal{C}}(B,D)$, then $h\circ f \circ g \in \mathcal{I}(C,D)$.
\end{enumerate}
\end{definition}

\begin{example}
Let $(A,d)$  be a differential graded $K$-algebra  and we consider the dg category $\mathcal{C}_{A}$ associated with the  $K$-algebra. Then  $\mathcal{I}$ is a dg-ideal of $A$ if and only if $\mathcal{I}$ is a dg-ideal of the  dg-category $\mathcal{C}_{A}$. 
\end{example}

Now, we recall the construction of the quotient category.
Let $\mathcal{I}$ be a dg-ideal in dg-category $\mathcal{C}$. The $\textbf{quotient category}$ $\mathcal{C}/\mathcal{I}$ is defined as follows: it has the same objects as $\mathcal{C}$ and $\mathrm{Hom}_{\mathcal{C}/\mathcal{I}}(A,B):=\frac{\mathrm{Hom}_{\mathcal{C}}(A,B)}{\mathcal{I}(A,B)}$ for each $A,B\in \mathcal{C}/\mathcal{I}$.\\
For $\overline{f}=f+\mathcal{I}(A,B)\in \mathrm{Hom}_{\mathcal{C}/\mathcal{I}}(A,B)$ and $\overline{g}=g+\mathcal{I}(B,C)\in \mathrm{Hom}_{\mathcal{C}/\mathcal{I}}(B,C)$ we set
$$\overline{g}\circ \overline{f}:=gf+\mathcal{I}(A,C)\in \mathrm{Hom}_{\mathcal{C}/\mathcal{I}}(A,C).$$ 
Let  $\mathcal{I}$ be a dg-ideal of $\mathcal{C}$, we have the canonical functor $\pi:\mathcal{C}\longrightarrow \mathcal{C}/\mathcal{I}$ defined as:  $\pi(A)=A$ $\forall A\in \mathcal{C}$ and $\pi(f):=\overline{f}=f+\mathcal{I}(A,B)\in \mathrm{Hom}_{\mathcal{C}/\mathcal{I}}(A,B)$  $\forall f\in \mathrm{Hom}_{\mathcal{C}}(A,B)$.

\begin{proposition}
Let $\mathcal{C}$ be a  dg-category and $\mathcal{I}$ a  differential graded ideal of $\mathcal{C}$. Then $\mathcal{C}/\mathcal{I}$ is a dg-category and $\pi:\mathcal{C}\longrightarrow \mathcal{C}/\mathcal{I}$ is a dg-functor.
\end{proposition}
\begin{proof}
Since $\mathcal{I}(A,B)$ is a dg $K$-submodule of  $\mathrm{Hom}_{\mathcal{C}}(A,B)$ we get that
$$\mathrm{Hom}_{\mathcal{C}/ \mathcal{I}}(A,B)=\frac{\mathrm{Hom}_{\mathcal{C}}(A,B)}{\mathcal{I}(A,B)}= \bigoplus_{n \in \mathbb{Z}} \Bigg(\frac{\mathrm{Hom}_{\mathcal{C}}(A,B)}{\mathcal{I}(A,B)}\Bigg)^{n}$$
where $\Big(\frac{\mathrm{Hom}_{\mathcal{C}}(A,B)}{\mathcal{I}(A,B)}\Big)^{n}= \pi (\mathrm{Hom}_{\mathcal{C}}^{n}(A,B))$
 with $\pi:\mathrm{Hom}_{\mathcal{C}}(A,B) \longrightarrow \frac{\mathrm{Hom}_{\mathcal{C}}(A,B)}{\mathcal{I}(A,B)}$  the canonical projection. Since $\mathcal{I}(A,B)$ is a dg-$K$-submodule of $\mathrm{Hom}_{\mathcal{C}}(A,B)$ we have that $d_{\mathcal{I}(A,B)}=d_{\mathcal{C}(A,B)}|_{\mathcal{I}(A,B)}$. Then we have that there exists a unique morphism
 $$d_{\frac{\mathcal{C}(A,B)}{\mathcal{I}(A,B)}}:\frac{\mathcal{C}(A,B)}{\mathcal{I}(A,B)} \longrightarrow \frac{\mathcal{C}(A,B)}{\mathcal{I}(A,B)}$$
such that the following diagram commutes 
\begin{equation}\label{difencociente}
\xymatrix{ 0\ar[r] & \mathcal{I}(A,B)\ar[r]\ar[d]^{d_{\mathcal{I}(A,B)}}  & \mathcal{C}(A,B)\ar[r]\ar[d]^{d_{\mathcal{C}(A,B)}} &  \frac{\mathcal{C}(A,B)}{\mathcal{I}(A,B)}\ar[r]\ar[d]^{d_{\frac{\mathcal{C}(A,B)}{\mathcal{I}(A,B)}}} & 0\\
0\ar[r] & \mathcal{I}(A,B)\ar[r]  & \mathcal{C}(A,B)\ar[r] &  \frac{\mathcal{C}(A,B)}{\mathcal{I}(A,B)}\ar[r] & 0.}
\end{equation}
Since $d_{\mathcal{C}(A,B)}\circ d_{\mathcal{C}(A,B)}=0$ we conclude that $d_{\frac{\mathcal{C}(A,B)}{\mathcal{I}(A,B)}}\circ d_{\frac{\mathcal{C}(A,B)}{\mathcal{I}(A,B)}}=0$. This tell us that 
$$d_{\frac{\mathcal{C}(A,B)}{\mathcal{I}(A,B)}}(\overline{f}):=\overline{(d_{\mathcal{C}(A,B)}(f))}$$
for each $\overline{f} \in  (\mathcal{C}/\mathcal{I})(A,B)$. Now, let $\overline{f}=f+\mathcal{I}(A,B)\in \mathrm{Hom}_{\mathcal{C}/\mathcal{I}}(A,B)$ and $\overline{g}=g+\mathcal{I}(B,C)\in \mathrm{Hom}_{\mathcal{C}/\mathcal{I}}(B,C)$ homogenous elements. In the following, by abuse of notation the letter $d$ will denote the different differentials in different hom-spaces. Then we have that
\begin{align*}
d(\overline{g}\circ \overline{f})=d(\overline{g\circ f})  =\overline{d(g\circ f)}
& =\overline{d(g)\circ f+(-1)^{|g|}g\circ d(f)}\\
& =\overline{d(g)\circ f}+(-1)^{|g|}\overline{g\circ d(f)}\\
& =\overline{d(g)}\circ\overline{f}+(-1)^{|g|}\overline{g}\circ\overline{d(f)}\\
& =d(\overline{g})\circ\overline{f}+(-1)^{|\overline{g}|}\overline{g}\circ d(\overline{f}).
\end{align*}
This implies that the composition $-\circ-:\frac{\mathcal{C}}{\mathcal{I}}(B,C)\otimes \frac{\mathcal{C}}{\mathcal{I}}(A,B)\longrightarrow \frac{\mathcal{C}}{\mathcal{I}}(A,C)$
is a morphism of degree zero in $\mathrm{DgMod}(K)$ which commutes with the differentials. Hence we conclude that $\mathcal{C}/\mathcal{I}$ is a dg-category. Now, since  $\Big(\frac{\mathrm{Hom}_{\mathcal{C}}(A,B)}{\mathcal{I}(A,B)}\Big)^{n}=\pi (\mathrm{Hom}_{\mathcal{C}}^{n}(A,B))$ and by diagram \ref{difencociente} we conclude that $\pi$ is a dg-functor.
\end{proof}

Let $\mathcal{C}$ be a dg-category. A full subcategory $\mathcal{B}$ of $\mathcal{C}$ is a dg-category and in this case we say that $\mathcal{B}$ is a full differential graded subcategory of $\mathcal{C}$.
Let $\mathcal{B}$  be a full differential graded subcategory of a dg-category $\mathcal{C}$  and consider  the following subset of  $\mathrm{Hom}_{\mathcal{C}}(C,C')^{n}$:
$$\mathcal{I}_{\mathcal{B}}(C,C')^{n}:= \left\{ f \in \mathcal{C}(C,C')^{n} \big\vert \exists \,\, B \in \mathcal{B},\,\, g \in \mathcal{C}(C,B),\,\, h\in \mathcal{C}(B,C')\,\,\text{with}\,\,f= h\circ  g \right\}.$$
That is,  $I_{\mathcal{B}}(C,C')^{n}$ consists of the set of homogenous morphisms that can be factorized through an object in $\mathcal{B}$.
We define 
$$\mathcal{I}_{\mathcal{B}}(C,C'):=\bigoplus_{n\in \mathbb{Z}}\mathcal{I}_{\mathcal{B}}(C,C')^{n}.$$

\begin{lemma}\label{I B ideal dg }
Let $\mathcal{C}$ be a dg-category with finite coproducts  and $\mathcal{B}$ a full differential graded subcategory  closed under coproducts. Then  $\mathcal{I}_{\mathcal{B}}$ is a  differential graded ideal of $\mathcal{C}$.
\end{lemma}
\begin{proof}
Let us see that $\mathcal{I}_{\mathcal{B}}(C,C')$ is an abelian subgroup of $\mathrm{Hom}_{\mathcal{C}}(C,C')$. Indeed, let $f_{1},f_{2}\in \mathcal{I}_{\mathcal{B}}(C,C')^{n}$.
Then there exist, $B_{1},B_{2}\in \mathcal{B}$ and morphisms $g_{1}:C\longrightarrow B_{1}$, $h_{1}:B_{1}\longrightarrow C'$, $g_{2}:C\longrightarrow B_{2}$, $h_{2}:B_{2}\longrightarrow C'$ such that
$f_{1}=h_{1}g_{1}$ and $f_{2}=h_{2}g_{2}$.
Consider the following diagram
$$\xymatrix{C\ar[rr]\ar[dr]_{a} & & C'\\
& B_{1}\oplus B_{2}\ar[ur]_{b}}$$
with $a=\left [\begin{smallmatrix} g_{1}
\\ -g_{2} \end{smallmatrix} \right ]$ and 
$b=\left [\begin{smallmatrix}  h_{1} & h_{2}\end{smallmatrix} \right ]$. Hence
$b\circ a=\left [\begin{smallmatrix}  h_{1} & h_{2}\end{smallmatrix} \right ]\left [\begin{smallmatrix} g_{1}
\\ -g_{2} \end{smallmatrix} \right ]=h_{1}g_{1}-h_{2}g_{2}=f_{1}-f_{2}$.
Since $B_{1}\oplus B_{2}\in \mathcal{B}$, we conclude that $f_{1}-f_{2}\in  \mathcal{I}_{\mathcal{B}}(C,C')^{n}$.\\
Now, if $f_{1}\in \mathcal{I}_{\mathcal{B}}(C,C')^{n}$ and let $f_{2}\in \mathcal{I}_{\mathcal{B}}(C,C')^{m}$ with $n\neq m$, we have that
$f_{1}-f_{2}\in \mathcal{I}_{\mathcal{B}}(C,C'):=\bigoplus_{n\in \mathbb{Z}}\mathcal{I}_{\mathcal{B}}(C,C')^{n}.$ We conclude that $\mathcal{I}_{\mathcal{B}}(C,C')$ is an abelian subgroup of $\mathrm{Hom}_{\mathcal{C}}(C,C')$.
Now, let $\lambda\in K$ and $f_{1}\in \mathcal{I}_{\mathcal{B}}(C,C')^{n}$. It is easy to see that $\lambda f_{1}\in \mathcal{I}_{\mathcal{B}}(C,C')^{n}$. Proving that
$\mathcal{I}_{\mathcal{B}}(C,C')$ is a graded $K$-submodule of $\mathrm{Hom}_{\mathcal{C}}(C,C')$.\\
Now, let us see that the differential $d_{\mathrm{Hom}_{\mathcal{C}}(C,C')}:\mathrm{Hom}_{\mathcal{C}}(C,C') \longrightarrow \mathrm{Hom}_{\mathcal{C}}(C,C')$ restricts  to the $K$-submodule $\mathcal{I}_{\mathcal{B}}(C,C')$. In order to simplify the notation we set $d:=d_{\mathrm{Hom}_{\mathcal{C}}(C,C')}$.\\
Since $d$ is a morphism of abelian groups and $\mathcal{I}_{\mathcal{B}}(C,C')$ is a subgroup of $\mathrm{Hom}_{\mathcal{C}}(C,C')$ it is enough to see that if $f \in \mathcal{I}_{\mathcal{B}}(C,C')^{n}$ then $d(f)\in \mathcal{I}_{\mathcal{B}}(C,C')$. Indeed, 
let $f \in \mathcal{I}_{\mathcal{B}}(C,C')^{n}$ be, then there exist, $B\in \mathcal{B}$ and  $g:C\longrightarrow B$, $h:B\longrightarrow C'$ such that
$f=hg$.\\
We have that $h=\sum_{i=1}^{m}h_{i}$ with $h_{i}\in \mathrm{Hom}_{\mathcal{C}}^{\alpha_{i}}(B,C')$ and $g=\sum_{j=1}^{r}g_{j}$ with $g_{i}\in \mathrm{Hom}^{\beta_{i}}_{\mathcal{C}}(C,B)$. Hence
$$f=(\sum_{i=1}^{m}h_{i})\circ (\sum_{j=1}^{r}g_{j})=\sum_{i,j} h_{i}g_{j}$$
with $h_{i}g_{j}\in \mathrm{Hom}^{\alpha_{i}+\beta_{j}}_{\mathcal{C}}(C,C')$.  By grouping all the terms that have the same grade we can write
$f= \sum_{l=1}^{p}f_{l}$ where $f_{l}\in \mathrm{Hom}^{\gamma_{l}}_{\mathcal{C}}(C,C')$ and
$f_{l}:=\sum_{\alpha_{i}+\beta_{j}=\gamma_{l}} h_{i}g_{j}$  and $\gamma_{l}\neq \gamma_{l'}$ if $l\neq l'$. Since $\mathrm{Hom}_{\mathcal{C}}(C,C')=\bigoplus_{n\in \mathbb{Z}}\mathrm{Hom}^{n}_{\mathcal{C}}(C,C')$, we conclude that $f_{l}= 0$ if $\gamma_{l}\neq n$. Hence
$$f=\sum_{\alpha_{i}+\beta_{j}=n}h_{i}g_{j}.$$
In order to see that $d(f)\in \mathcal{I}_{\mathcal{B}}(C,C'):=\bigoplus_{n\in \mathbb{Z}}\mathcal{I}_{\mathcal{B}}(C,C')^{n}$ it is enough to see that
$d( h_{i}g_{j})\in \mathcal{I}_{\mathcal{B}}(C,C'):=\bigoplus_{n\in \mathbb{Z}}\mathcal{I}_{\mathcal{B}}(C,C')^{n}$ if $\alpha_{i}+\beta_{j}=n$. Indeed, 

 \begin{align*}
 d_{\mathrm{Hom}_{\mathcal{C}}(C,C')}(h_{i}g_{j}) & = d_{\mathrm{Hom}_{\mathcal{C}}(B,C')}(h_{i})\circ g_{j} +(-1)^{|h_{i}|}h_{i}\circ d_{\mathrm{Hom}_{\mathcal{C}}(C,B)}(g_{j}).
 \end{align*}
 We have that $d_{\mathrm{Hom}_{\mathcal{C}}(B,C')}(h_{i})\circ g_{j}$, $(-1)^{|h_{i}|}h_{i} \circ d_{\mathrm{Hom}_{\mathcal{C}}(C,B)}(g_{j})$ are morphisms  that  factors  through  $B \in \mathcal{B}$. Thus, $d_{\mathrm{Hom}_{\mathcal{C}}(B,C')}(h_{i})\circ g_{j}$, $(-1)^{|h_{i}|}h_{i} \circ d_{\mathrm{Hom}_{\mathcal{C}}(C,B)}(g_{j})$   are elements of  
$\mathcal{I}_{\mathcal{B}}(C,C')$  and we conclude that  $d_{\mathrm{Hom}_{\mathcal{C}}(C,C')}(h_{i}g_{j})\in \mathcal{I}_{\mathcal{B}}(C,C')$.\\
Then, we have the differential $d_{\mathrm{Hom}_{\mathcal{C}}(C,C')}:\mathrm{Hom}_{\mathcal{C}}(C,C') \longrightarrow \mathrm{Hom}_{\mathcal{C}}(C,C')$ restricts  to the $K$-submodule $\mathcal{I}_{\mathcal{B}}(C,C')$ and hence $\mathcal{I}_{\mathcal{B}}(C,C')$ is a dg-$K$-submodule of $\mathrm{Hom}_{\mathcal{C}}(C,C')$.\\
Let $f\in \mathcal{I}_{\mathcal{B}}(C,C')^{n}$ and $g\in \mathcal{C}(A,C)^{p}$  and
$h\in \mathcal{C}(C',B)^{q}$. 
We conclude that $hfg\in \mathcal{I}_{\mathcal{B}}(A,B)^{n+p+q}$. Since every element of $\mathrm{Hom}_{\mathcal{C}}(X,Y)$ is a sum of homogenous elements for all $X,Y\in \mathcal{C}$, we conclude that  If $f \in \mathcal{I}_{\mathcal{B}}(A,B)$, $g \in \mathrm{Hom}_{\mathcal{C}}(C,A)$ and $h \in \mathrm{Hom}_{\mathcal{C}}(B,D)$, then $h\circ f \circ g \in \mathcal{I}_{\mathcal{B}}(C,D)$, proving that $\mathcal{I}_{\mathcal{B}}$ is a dg-ideal of $\mathcal{C}$.
\end{proof}
Recall that given a functor $F:\mathcal{A}\longrightarrow \mathcal{B}$ between arbitrary categories  the $\textbf{essential}$ $\textbf{image}$ of $F$ denoted by $\mathrm{Im}(F)$ is the full subcategory of $\mathcal{B}$ whose objects are defined as: $\mathrm{Im}(F):=\{B\in \mathcal{B}\mid \text{exist}\,\,A\in \mathcal{A}\,\,\text{such that}\,\, F(A)\simeq B\}.$
If $\mathcal{B}$ has zero object,  we define the $\textbf{kernel}$ of $F$ as the full subcategory of $\mathcal{A}$ whose objects are given as follows $\mathrm{Ker}(F):=\{A\in \mathcal{A}\mid F(A)\simeq 0\}$.\\
Let $\mathcal{C}$ be a dg-category with finite coproducts  and $\mathcal{B}$ a full differential graded subcategory  closed under coproducts.
We have the canonical functor of dg-categories $\pi:\mathcal{C}\longrightarrow \mathcal{C}/\mathcal{I}_{\mathcal{B}}.$ Hence we have the induced dg-functor $\pi_{\ast}:\mathrm{DgMod}(\mathcal{C}/\mathcal{I}_{\mathcal{B}})\longrightarrow \mathrm{DgMod}(\mathcal{C}).$ Let $j:\mathcal{B}\longrightarrow \mathcal{C}$ be inclusion and consider the induced $j_{\ast}: \mathrm{DgMod}(\mathcal{C})\longrightarrow \mathrm{DgMod}(\mathcal{B}).$ 

\begin{proposition}\label{Ker=Im}
We have that $\mathrm{Ker}(j_{\ast})=\mathrm{Im}(\pi_{\ast})$.
\end{proposition}
\begin{proof}
Let $F\in \mathrm{DgMod}(\mathcal{C})$ such that $j_{\ast}(F)=0$. That is, 
$F(B)=0$ for all $B\in \mathcal{B}$. Then we have that $F(f)=0$ for all $f\in \mathcal{I}_{B}(C,C')$. Then there exists a morphism $\overline{F}_{C,C'}$
 such that the following diagram commutes in $\mathrm{DgMod}_{Str}(K)$:
$$\xymatrix{0\ar[r] & \mathcal{I}_{\mathcal{B}}(C,C')\ar[r] & \mathcal{C}(C,C')\ar[r]\ar[dr]_{F_{C,C'}} & \frac{\mathcal{C}(C,C')}{\mathcal{I}_{\mathcal{B}}(C,C')}\ar[r]\ar[d]^{\overline{F}_{C,C'}} & 0\\
 & & & [F(C),F(C')]}$$
Then we define a functor $\overline{F}:\mathcal{C}/\mathcal{I}_{\mathcal{B}}\longrightarrow \mathrm{DgMod}(K)$ as follows:
\begin{enumerate}
\item [(a)] $\overline{F}(C):=F(C)$ for all $C\in \mathcal{C}/\mathcal{I}_{\mathcal{B}}$
\item [(b)] If $\overline{f}=f+\mathcal{I}_{\mathcal{B}}(C,C')\in \mathcal{C}/\mathcal{I}_{\mathcal{B}}(C,C')$ we set $\overline{F}(\overline{f})=\overline{F}_{C,C'}(\overline{f})=F(f).$
\end{enumerate}
We also have that if $\overline{f}\in  (\mathcal{C}/\mathcal{I}_{\mathcal{B}}(C,C'))^{n}$ then $\overline{F}(\overline{f})\in [\overline{F}(C),\overline{F}(C')]^{n}$.
Since $F$ is a dg-functor we have that $F$ commutes with differentials. Let us see that the following diagram commutes for homogenous elements
\begin{equation}
\xymatrix{\frac{\mathcal{C}(C,C')}{\mathcal{I}_{\mathcal{B}}(C,C')}\ar[r]^{\overline{F}}\ar[d]^{d_{\mathcal{C}/\mathcal{I}}} & [F(C),F(C')]\ar[d]^{D}\\
\frac{\mathcal{C}(C,C')}{\mathcal{I}_{\mathcal{B}}(C,C')}\ar[r]_{\overline{F}} & [F(C),F(C')]}
\end{equation}
Indeed, let $\overline{f}=f+\mathcal{I}_{\mathcal{B}}(C,C')\in \Big(\frac{\mathcal{C}(C,C')}{\mathcal{I}_{\mathcal{B}}(C,C')}\Big)^{n}=\mathcal{C}(C,C')^{n}+\mathcal{I}_{\mathcal{B}}(C,C')$ with $f\in \mathcal{C}(C,C')^{n}$. Then 
\begin{align*}
D\Big(\overline{F}(\overline{f})\Big) =D(F(f)) =F(d_{\mathcal{C}}(f)) =\overline{F}(\overline{d_{\mathcal{C}}(f)}) =\overline{F}(d_{\mathcal{C}/\mathcal{I}}(\overline{f})).
\end{align*}
Then the required diagram commutes. Then we have that $\overline{F}$ is a dg-functor and $F=\overline{F}\circ \pi=\pi_{\ast}(\overline{F})$. Proving that $\mathrm{Ker}(j_{\ast})\subseteq \mathrm{Im}(\pi_{\ast})$.\\
Let  $M\in \mathrm{DgMod}(\mathcal{C}/\mathcal{I}_{\mathcal{B}})$. Then $\pi_{\ast}(M):=M\circ \pi\in \mathrm{Mod}(\mathcal{C})$ and for $1_{B}$  with $B\in \mathcal{B}$ have that $1_{(M\pi)(B)}=(M\circ \pi)(1_{B})=M(\overline{1_{\mathcal{B}}})=M(0)=0$. This proves that $(M\pi)(B)=0$ and hence $\pi_{\ast}(M)\in \mathrm{Ker}(j_{\ast})$, proving that $\mathrm{Im}(\pi_{\ast})\subseteq  \mathrm{Ker}(j_{\ast})$.
\end{proof}

\begin{proposition}\label{inducedgfull}
Let $\mathcal{C}$ be a dg-category with finite coproducts  and $\mathcal{B}$ a full differential graded subcategory  closed under coproducts. Let $\pi:\mathcal{C}\longrightarrow \mathcal{C}/\mathcal{I}_{\mathcal{B}}$. Then $\pi_{\ast}:\mathrm{DgMod}(\mathcal{C}/\mathcal{I}_{\mathcal{B}})\longrightarrow \mathrm{DgMod}(\mathcal{C})$ is dg-fully faithful (see Definition \ref{dgadjuncion}).
\end{proposition}
\begin{proof}
For $M,N\in \mathrm{DgMod}(\mathcal{C}/\mathcal{I}_{\mathcal{B}})$ we get that:
\begin{align*}
\mathrm{DgNat}(M,N) \!\!= \!\!\int_{C\in \mathcal{C}/\mathcal{I}_{\mathcal{B}}}\!\!\!\![M(C),N(C)] =\int_{C\in \mathcal{C}}\!\!\!\![M(C),N(C)] & =\int_{C\in \mathcal{C}}\!\!\!\![M\pi(C),N\pi(C)] \\
& =\mathrm{DgNat}(M\pi,N\pi).
\end{align*}
\end{proof}

\begin{proposition}\label{adjointfull} For any dg-functor $F:\mathcal{A} \longrightarrow \mathcal{B},$ the following statements hold.
\begin{enumerate}
\item [(a)] Let $G:\mathcal{B} \longrightarrow \mathcal{A}$ be a dg-left adjoint of $F.$ Then $F$ is dg-fully faithful if and only if the counit $\varepsilon:G\circ F\longrightarrow 1_{\mathcal{A}}$ is an isomorphism.

\item [(b)] Let $H:\mathcal{B} \longrightarrow \mathcal{A}$ be a dg-right adjoint of $F.$ Then $F$ is dg-fully faithful if and only if the unit $\eta:1_{\mathcal{A}}\longrightarrow H\circ F$ is an isomorphism.
\end{enumerate}
\end{proposition}
\begin{proof} 
The proof given in \cite[Theorem 3.4.1]{Borceux1} on page 114, can be adapted to this setting.
\end{proof}

\begin{proposition}\label{dgfulliso}
Let $F:\mathcal{A}\longrightarrow \mathcal{B}$ and $G:\mathcal{B}\longrightarrow \mathcal{A}$  be dg-functors such that $F$ is dg-left adjoint to $G$. Then $G\circ F\simeq 1_{\mathcal{A}}$ if and only if the unit $\eta:1_{\mathcal{A}}\longrightarrow GF$ is an isomorphism. Dually $FG\simeq 1_{\mathcal{B}}$ if and only if the counit $\epsilon:FG\longrightarrow  1_{\mathcal{B}}$ is an isomorphisms.  In particular, $F$ is dg-fully faithful if and only if $G\circ F\simeq 1_{\mathcal{A}}$ and $G$ is dg-fully faithful if and only if $FG\simeq 1_{\mathcal{B}}$. 
\end{proposition}
\begin{proof}
The proof given \cite[Lemma 1.3]{Johnstone1} can be adapted to this setting.
\end{proof}

\begin{proposition}\label{LemmaMoerdik}
Consider the following diagram of dg-functors
$$\xymatrix{\mathcal{A}\ar[rrr]|{G}  &  &&\mathcal{B}
\ar@<-2ex>[lll]_{F}\ar@<2ex>[lll]^{H}}$$
such that $F$ is dg-left adjoint to $G$ and $G$ is dg-left adjoint to $H$. Then $F$ is  dg-fully faithful if and only if $H$  is dg-fully faithful.
\end{proposition}
\begin{proof}
The proof given in \cite[Lemma 1]{MacMoer} in p. 369 in Section VII.4 in Mac Lane-Moerdijk's book, can be adapted to this setting.
\end{proof}

\begin{proposition}\label{unacomposicion}
Let $F:\mathcal{A}\longrightarrow \mathcal{C}$ be a full and faithful dg-functor.
Then $\mathrm{Lan}_{F}(M)\circ F\simeq M$ for all $M\in \mathrm{DgMod}(\mathcal{A})$. That is $F_{\ast}\circ \mathbb{L}\mathrm{an}_{F}(-)\simeq 1_{\mathrm{DgMod}(\mathcal{A})}.$
\end{proposition}
\begin{proof}
We have the isomorphisms for $N\in \mathrm{DgMod}(\mathcal{A})$:

\begin{align*}
& \mathrm{DgNat}\Big(\mathrm{Lan}_{F}(M)\circ F,N\Big)\simeq \\
&= \int_{A\in \mathcal{A}}\Big[\mathrm{Lan}_{F}(M)(F(A)),N(A)\Big] & [\text{Proposition}\,\, \ref{construcciondgnat}]\\
 & =\int_{A\in \mathcal{A}}\Big[\mathcal{C}(F,F(A))\otimes_{\mathcal{A}}M ,N(A)\Big] & [\text{Proposition}\,\, \ref{existeLan}]
 \\
 & = \int_{A\in \mathcal{A}}\Big[\int^{A'\in \mathcal{A}} \mathcal{C}(F(A'),F(A))\otimes M(A'),N(A)\Big] & [\text{Definition}\,\, \ref{tensorprodascoend}]\\
 & = \int_{A\in \mathcal{A}} \int_{A'\in \mathcal{A}} \Big[\mathcal{C}(F(A'),F(A))\otimes M(A'),N(A)\Big]  & \quad [(\ast)]\\
& = \int_{A\in \mathcal{A}} \int_{A'\in \mathcal{A}} \Big[\mathcal{C}(F(A'),F(A)), [M(A'),N(A)]\Big] &  [\text{adjointness}]\\
& = \int_{A\in \mathcal{A}} \int_{A'\in \mathcal{A}} \Big[\mathcal{A}(A',A), [M(A'),N(A)]\Big] & [F\,\, \text{is full and faithful}]\\
& = \int_{A\in \mathcal{A}} \int_{A'\in \mathcal{A}} \Big[\mathcal{A}(A',A)\otimes M(A'), N(A)]\Big] & [\text{adjointness}]\\
& = \int_{A\in \mathcal{A}}  \Big[\int^{A'\in \mathcal{A}}\mathcal{A}(A',A)\otimes M(A'), N(A)]\Big] &  [(\ast)]\\
& = \int_{A\in \mathcal{A}}  \Big[\mathcal{A}(-,A)\otimes_{\mathcal{A}}M, N(A)]\Big] &  [\text{Definition}\,\, \ref{tensorprodascoend}]\\
& = \int_{A\in \mathcal{A}}  \Big[M(A), N(A)]\Big] & [\text{Corollary}\,\, \ref{tensorhom}]\\
& =\mathrm{DgNat}(M,N) & [\text{Proposition}\,\, \ref{construcciondgnat}]
\end{align*}
where the isomorphisms $(\ast)$ follows from the fact that  $[-,N(A)]$ changes ends for coends, and this is the dual to formula 2.3 on page 28  in \cite{Kelly}. This implies by Yoneda's Lemma that $\mathrm{Lan}_{F}(M)\circ F\simeq M$.
\end{proof}

\begin{theorem}\label{dg-recollement}
Let $\mathcal{C}$ be a dg-category with finite coproducts  and $\mathcal{B}$ a full differential graded subcategory  closed under coproducts. Consider the inclusion $j:\mathcal{B}\longrightarrow \mathcal{C}$ the projection $\pi:\mathcal{C}\longrightarrow \mathcal{C}/\mathcal{I}_{\mathcal{B}}$ and the induced functors $\pi_{\ast}:\mathrm{DgMod}(\mathcal{C}/\mathcal{I}_{\mathcal{B}})\longrightarrow \mathrm{DgMod}(\mathcal{C})$ and
$j^{\ast}:\mathrm{DgMod}(\mathcal{C})\longrightarrow \mathrm{DgMod}(\mathcal{B})$. Then, we have the following diagram of dg-adjoint triples
$$\xymatrix{\mathrm{DgMod}(\mathcal{C}/\mathcal{I}_{\mathcal{B}})\ar[rr]|{\pi_{\ast}=\pi_{!}}  &  & \mathrm{DgMod}(\mathcal{C})\ar[rr]|{j^{!}=j^{\ast}}\ar@<-2ex>[ll]_{\pi^{\ast}}\ar@<2ex>[ll]^{\pi^{!}}  &   & \mathrm{DgMod}(\mathcal{B})\ar@<-2ex>[ll]_{j_{!}}\ar@<2ex>[ll]^{j_{\ast}}}$$

where $\pi^{\ast}=\mathbb{L}\mathrm{an}_{\pi}(-)$, $\pi^{!}=\mathbb{R}\mathrm{an}_{\pi}(-)$  and $j_{!}=\mathbb{L}\mathrm{an}_{j}(-)$, $j_{\ast}=\mathbb{R}\mathrm{an}_{j}(-)$. Moreover we have that 
\begin{enumerate}
\item [(a)] $\mathrm{Im}(\pi_{\ast})=\mathrm{Ker}(j{\ast})$,

\item [(b)] $\pi_{\ast}$ and $j_{!}$ and $j_{\ast}$ are dg-fully faithful.
\end{enumerate}
\end{theorem}
\begin{proof}
By Proposition \ref{dgadjuninduced}, we have the above diagram of dg-adjoints. By Proposition \ref{inducedgfull}, we have that $\pi_{\ast}$ is dg-full faithful. Since $j:\mathcal{B}\longrightarrow \mathcal{C}$ is dg-fully faithfull, by Proposition \ref{unacomposicion}, we have that $j_{!}$ is dg-fully faithful; and by Proposition \ref{LemmaMoerdik} we conclude that $j_{\ast}$ is dg-fully faithful. Finally, by Proposition \ref{Ker=Im}, we have that $\mathrm{Im}(\pi_{\ast})=\mathrm{Ker}(j{\ast})$.
\end{proof}

\subsection{Recollement of abelian categories}
Now, we want to obtain a similar result of Theorem \ref{dg-recollement} of the underlying category of $\mathrm{DgFun}(\mathcal{A},\mathrm{DgMod}(K))$. First, we have the following definitions.

\begin{Remark}\label{morcatsubya}
We recall that given $\mathcal{A}$  a $\mathrm{dg}$-category we have the underlying category $\mathcal{A}_{0}$ whose objects are the same as the objects of $\mathcal{A}$ and whose morphisms are
$$\mathrm{Hom}_{\mathcal{A}_{0}}(X,Y):=\mathrm{Hom}_{\mathrm{DgMod}_{Str}(K)}(K,\mathcal{A}(X,Y)).$$
Let $f:K\longrightarrow \mathcal{A}(X,Y))$ a morphism in $\mathrm{DgMod}_{Str}(K)$.  Then we have that
$f$ is a morphism of degree zero and the following diagram commutes
$$\xymatrix{K\ar[r]^{f}\ar[d]_{d_{K}} &  \mathcal{A}(X,Y)\ar[d]^{d}\\
K\ar[r]_{f} & \mathcal{A}(X,Y)}$$
Since $f$ is a morphism of $K$-modules we have that $f$ is determined by $f(1)$. Now, since $d_{K}=0$ we have that
$d (f(1))=f(d_{K}(1))=f(0)=0$. Moreover, since $f$ is of degree zero and $1\in K^{0}$ we have that $f(1)\in \mathcal{A}(X,Y)^{0}$ (homogenous component of degree zero of $\mathcal{A}(X,Y)$). Hence, we have an isomorphism

$$\mathrm{Hom}_{\mathcal{A}_{0}}(X,Y)\simeq \{\alpha \in \mathcal{A}(X,Y)^{0}\mid d_{\mathcal{A}(X,Y)}(f)=0\}.$$
\end{Remark}

\begin{proposition}\label{todopasasubyacentes}
Let $F:\mathcal{A}\longrightarrow \mathcal{B}$  and $G:\mathcal{B}\longrightarrow \mathcal{A}$ be  dg-functors. Then we have induced functors
$F_{0}:\mathcal{A}_{0}\longrightarrow \mathcal{B}_{0}$  and $G_{0}:\mathcal{B}_{0}\longrightarrow \mathcal{A}_{0}$ between the its underlying categories. Moreover
\begin{enumerate}
\item [(a)] If $F$ is dg-left adjoint to $G$, then $F_{0}$ is left adjoint to $G_{0}$.

\item [(b)] If $F$ is dg-full and faithfull, then $F$ is full and fatifull.
\end{enumerate}
\end{proposition}
\begin{proof}
It is straightforward.
\end{proof}

\begin{definition}\label{catciclo1}
Let $\mathcal{A}$ be a dg-category. Consider the dg category $\mathrm{DgMod}(\mathcal{A})=\mathrm{DgFun}(\mathcal{A},\mathrm{DgMod}(K))$ we denote by $Z^{0}(\mathrm{DgMod}(\mathcal{A}))$ its underlying category
$$Z^{0}(\mathrm{DgMod}(\mathcal{A})):=\Big(\mathrm{DgMod}(\mathcal{A})\Big)_{0}$$
defined as follows:
\begin{enumerate}
\item [(a)] $Z^{0}(\mathrm{DgMod}(\mathcal{A}))$ has the same objects as $\mathrm{DgMod}(\mathcal{A})$.

\item [(b)] Now, given $M,N\in Z^{0}(\mathrm{DgMod}(\mathcal{A}))$ we define
$$\mathrm{Hom}_{Z^{0}(\mathrm{DgMod}(\mathcal{A}))}(M,N):=\Big\{\alpha\in \mathrm{DgNat}^{0}(M,N)\mid d(\alpha)=0\Big\}.$$
\end{enumerate}
\end{definition}

\begin{Remark}\label{notacionkeller}
\begin{enumerate}
\item [(a)]  Acoording to the notation given in  \cite{Keller2} in p. 4 and 5, for a  dg category $\mathcal{A}$ we have that  $\mathcal{C}_{dg}(\mathcal{A})=\mathrm{DgMod}(\mathcal{A})$ and 
$\mathcal{C}(A):=Z^{0}(\mathcal{C}_{dg}(\mathcal{A}))=Z^{0}(\mathrm{DgMod}(\mathcal{A}))$.

\item [(b)] It is well known that $Z^{0}(\mathrm{DgMod}(\mathcal{A}))$ is an abelian category (see for example  p. 5  in  \cite{Keller2}).
\end{enumerate}
\end{Remark}

We recall the definition of recollement between abelian categories.

\begin{definition}
Let $\mathcal{A}$, $\mathcal{B}$ and $\mathcal{C}$ be abelian categories. Then the diagram

$$\xymatrix{\mathcal{B}\ar[rr]|{i_{\ast}=i_{!}}  &  &\mathcal{A}\ar[rr]|{j^{!}=j^{\ast}}\ar@<-2ex>[ll]_{i^{\ast}}\ar@<2ex>[ll]^{i^{!}}  &   &\mathcal{C}\ar@<-2ex>[ll]_{j_{!}}\ar@<2ex>[ll]^{j_{\ast}}}$$
is called a $\textbf{recollement}$, if the additive functors $i^{\ast},i_{\ast}=i_{!}, i^{!},j_{!},j^{!}=j^{*}$ and $j_{*}$ satisfy the following conditions:
\begin{itemize}
\item[(R1)] $(i^{*},i_{*}=i_{!},i^{!})$ and $(j_{!},j^{!}=j^{*},j_{*})$ are adjoint  triples, i.e. $(i^{*},i_{*})$, $(i_{!},i^{!})$  $(j_{!},j^{!})$ and $(j^{*},j_{*})$ are adjoint pairs;
\item[(R2)] $\mathrm{Ker}(j^{*})=\mathrm{Im}(i_{*})$;
\item[(R3)] $i_{*}, j_{!},j_{*}$ are full embedding functors.
\end{itemize}
\end{definition}

\begin{theorem}
Let $\mathcal{C}$ be a dg-category with finite coproducts  and $\mathcal{B}$ a full differential graded subcategory  closed under coproducts. Consider the inclusion $j:\mathcal{B}\longrightarrow \mathcal{C}$ and the projection $\pi:\mathcal{C}\longrightarrow \mathcal{C}/\mathcal{I}_{\mathcal{B}}$. Then we have a recollement of abelian categories
$$\xymatrix{Z^{0}(\mathrm{DgMod}(\mathcal{C}/\mathcal{I}_{\mathcal{B}}))\ar[rr]|{(\pi_{\ast})_{0}}  &  & Z^{0}(\mathrm{DgMod}(\mathcal{C}))\ar[rr]|{(j^{\ast})_{0}}\ar@<-2ex>[ll]_{(\pi^{\ast})_{0}}\ar@<2ex>[ll]^{(\pi^{!})_{0}}  &   & Z^{0}(\mathrm{DgMod}(\mathcal{B}))\ar@<-2ex>[ll]_{(j_{!})_{0}}\ar@<2ex>[ll]^{(j_{\ast})_{0}}}$$
\end{theorem}
\begin{proof}
It follows from Theorem \ref{dg-recollement} and Proposition \ref{todopasasubyacentes}.

\end{proof}

\section{Recollement in differential graded triangular matrix categories}
The goal of this section is to get  a generalization of a result given by Chen and Zheng in \cite[Theorem 4.4]{Chen}. So, firstly we recall the following definitions.
\begin{definition}\label{leftrigtrecol}
Let $\mathcal{A}$, $\mathcal{B}$ and $\mathcal{C}$ be dg-categories
\begin{itemize}
\item[(a)]  The diagram
$$\xymatrix{\mathcal{C}\ar@<-2ex>[rr]_{i_{\ast}} & & \mathcal{A}\ar@<-2ex>[rr]_{j^{!}}\ar@<-2ex>[ll]_{i^{\ast}} & &  \mathcal{B}\ar@<-2ex>[ll]_{j_{!}}}$$
is a called a  $\textbf{left}$ $\textbf{dg-recollement}$ if the dg-functors $i^{\ast},i_{\ast},j_{!}$ and $j^{!}$ satisfy the following conditions:
\begin{itemize}
\item[(LR1)] $(i^{\ast},i_{\ast})$ and $(j_{!},j^{!})$ are dg-adjoint pairs;
\item[(LR2)] $j^{!}i_{\ast}=0$;
\item[(LR3)] $i_{\ast},j_{!}$ are dg-fully faithful.
\end{itemize}

\item[(b)] The diagram
$$\xymatrix{\mathcal{C}\ar@<2ex>[rr]^{i_{!}} & & \mathcal{A}\ar@<2ex>[rr]^{j^{\ast}}\ar@<2ex>[ll]^{i^{!}} & &  \mathcal{B}\ar@<2ex>[ll]^{j_{\ast}}}$$
is  called a  $\textbf{right}$ $\textbf{dg-recollement}$ if the dg-functors $i_{!},i^{!},j^{\ast}$  and $j_{\ast}$ satisfy  the following 
conditions:
\begin{itemize}
\item[(RR1)] $(i_{!},i^{!})$ and $(j^{\ast},j_{\ast})$ are adjoint pairs;
\item[(RR2)] $j^{\ast}i_{!}=0$;
\item[(RR3)] $i_{!},j_{\ast}$ are dg-fully faithful.
\end{itemize}

\end{itemize}
\end{definition}

In \cite{PaezSandovalSan} the notion of differential graded triangular matrix category was introduced. For convenience of the reader, we recall briefly these concepts.  Let $\mathcal{R}$, $\mathcal{T}$ be  dg-categories and $M\in\mathrm{DgMod}(\mathcal R\otimes\mathcal {T}^{op})$,  the $\textbf{dg-triangular matrix}$ $\textbf{category}$  $\mathbf{\Lambda}=\left[ \begin{smallmatrix}
\mathcal{T} & 0 \\ M & \mathcal{R}
\end{smallmatrix}\right]$ is defined  as follows.

\begin{enumerate}
\item [(a)] The class of objects of this category are matrices $ \left[
\begin{smallmatrix}
T & 0 \\ M & R
\end{smallmatrix}\right]  $ where the objects $ T $ and $R$ are in  $\mathcal{T} $ and $\mathcal{R} $ respectively.

\item [(b)] Given a pair of objects
$\left[ \begin{smallmatrix}
T & 0 \\
M & R
\end{smallmatrix} \right] ,  \left[ \begin{smallmatrix}
T' & 0 \\
M & R'
\end{smallmatrix} \right]$ in
$\mathbf{\Lambda}$ , 

$$\mathrm{Hom}_{\mathbf{\Lambda}}\left (\left[ \begin{smallmatrix}
T & 0 \\
M & R
\end{smallmatrix} \right] ,  \left[ \begin{smallmatrix}
T' & 0 \\
M & R'
\end{smallmatrix} \right]  \right)  := \left[ \begin{smallmatrix}
\mathrm{Hom}_{\mathcal{T}}(T,T') & 0 \\
M(R',T) & \mathrm{Hom}_{\mathcal{R}}(R,R')
\end{smallmatrix} \right].$$
\end{enumerate}
The composition is given by
\begin{eqnarray*}
\circ&:&\left[  \begin{smallmatrix}
{\mathcal{T}}(T',T'') & 0 \\
M(R'',T') & {\mathcal{R}}(R',R'')
\end{smallmatrix}  \right] \times \left[
\begin{smallmatrix}
{\mathcal{T}}(T,T') & 0 \\
M(R',T) & {\mathcal{R}}(R,R')
\end{smallmatrix} \right]\longrightarrow\left[
\begin{smallmatrix}
{\mathcal{T}}(T,T'') & 0 \\
M(R'',T) & {\mathcal{R}}(R,R'')\end{smallmatrix} \right] \\
&& \left( \left[ \begin{smallmatrix}
t_{2} & 0 \\
m_{2} & r_{2}
\end{smallmatrix} \right], \left[
\begin{smallmatrix}
t_{1} & 0 \\
m_{1} & r_{1}
\end{smallmatrix} \right]\right)\longmapsto\left[
\begin{smallmatrix}
t_{2}\circ t_{1} & 0 \\
m_{2}\bullet t_{1}+r_{2}\bullet m_{1} & r_{2}\circ r_{1}
\end{smallmatrix} \right].
\end{eqnarray*}
We recall that for homogenous elements we have that $ m_{2}\bullet t_{1}:=(-1)^{|m_{2}|\cdot |t_{1}|}M(1_{R''}\otimes t_{1}^{op})(m_{2})$ and
$r_{2}\bullet m_{1}=M(r_{2}\otimes 1_{T})(m_{1})$,
and given 
an object $\left[
\begin{smallmatrix}
T & 0 \\
M & R
\end{smallmatrix} \right]\in \mathbf{\Lambda}$, the identity morphism is given by $1_{\left[
\begin{smallmatrix}
T & 0 \\
M & R
\end{smallmatrix} \right]}:=\left[
\begin{smallmatrix}
1_{T} & 0 \\
0 & 1_{R}
\end{smallmatrix} \right].$\\

In \cite[Theorem 3.19]{PaezSandovalSan} it is proved the following  result.

\begin{theorem}\label{equivalenceLEOS}
Let $\mathcal{R}$ and $\mathcal{T}$  be dg-categories and $ M\in \mathrm{DgMod}(\mathcal{R}\otimes \mathcal{T}^{op})$. Then,  there exists a dg-functor $\mathbb{G}_{1}:\mathrm{DgMod}(\mathcal{R})\longrightarrow \mathrm{DgMod}(\mathcal{T})$ for which there is an equivalence of  dg-categories 
$$\Big( \mathrm{DgMod}(\mathcal T), \mathbb{G}_{1}\mathrm{DgMod}(\mathcal R)\Big)\cong\mathrm{DgMod}\big(\left[ \begin{smallmatrix}
\mathcal T & 0 \\
M & \mathcal R
\end{smallmatrix} \right] \big).$$
\end{theorem}

In the remark below we briefly recall the definition of the functor $\mathbb{G}_{1}$ mentioned above. Next, we will construct other functor $\mathbb{G}_{2}$ in a similar way as for $\mathbb{G}_{1}$ and these functors will play an importan role in the proof of the main result of this section.

\begin{Remark}\label{defofG1}
Let $\mathcal{R,S,T}$ be dg-categories  and consider   $ M \in \mathrm{DgMod}(\mathcal{R}\otimes\mathcal{T}^{op})$. For all $ T\in \mathcal{T} $ we have the functor $ M_{T} :\mathcal{R}\longrightarrow \mathrm{DgMod}(K)$ defined as follows:

\begin{enumerate}
\item $M_{T}(R):=M(R,T) $, for all $ R\in \mathcal{R} $.
\item $M_{T}(r):=M(r\otimes 1_{T}): M_{T}(R)\longrightarrow M_{T}(R')$, for all $ r\in  \mathrm{Hom}_{\mathcal{R}}(R,R'). $
\end{enumerate}
Also for all $ t\in  \mathrm{Hom}_{\mathcal{T}}(T,T') $ homogeneous of degree $|t|$ we have a dg-natural transformation of degree $|t|$ (see Definition \ref{dgnaturaltrans}), $ \bar{t}:M_{T'}\longrightarrow M_{T} $ such that $ \bar{t}=\lbrace [\bar{t}]_{R}:M_{T'}(R)\longrightarrow M_{T}(R)\rbrace_{ R\in \mathcal{R}} $ where $ [\bar{t}]_{R}=M(1_{R}\otimes t^{op}):M(R,T')\longrightarrow M(R,T) $.\\
So we have the dg-functor $\mathbb{G}_{1}:\mathrm{DgMod}(\mathcal{R})\longrightarrow \mathrm{DgMod}(\mathcal{T}) $ as follows:
\begin{enumerate}
 \item For  $B \in \mathrm{DgMod}(\mathcal{R})$, $\mathbb{G}_{1}(B)(T)=\Ho_{\mathrm{DgMod}(\mathcal{R})}(M_{T},B)$ for all $T \in \T$. Now, for  $t: T \longrightarrow T'$ a homogeneous morphism in $\T$ we have $\mathbb{G}_{1}(B)(t):\Ho_{\mathrm{DgMod}(\mathcal{R})}(M_{T},B) \longrightarrow \Ho_{\mathrm{DgMod}(\mathcal{R})}(M_{T'},B)$  and  for each $\eta: M_{T} \longrightarrow B$  homogeneous morphism of degree $|\eta|$  we have $\mathbb{G}_{1}(B)(t)(\eta):= \Ho (\overline{t}, B)(\eta)$ $= (-1)^{|\eta||\overline{t}|} \eta \circ \overline{t}$.
 
 \item  Let $\varepsilon: B \longrightarrow B'$ be a  dg-natural transformation of dg $\mathcal{R}$-modules of degree $|\varepsilon|$. We assert that  $\mathbb{G}_{1}(\varepsilon):\mathbb{G}_{1}(B) \longrightarrow \mathbb{G}_{1}(B')$ is a dg-natural transformation in $\mathrm{DgMod}(\T)$ of degree $|\varepsilon|$. Indeed, for each  $T \in \T$ we have that $[\mathbb{G}_{1}(\varepsilon)]_{T}:= \Ho_{\mathrm{DgMod}(\mathcal{R})}(M_{T}, \varepsilon):  \Ho_{\mathrm{DgMod}(\mathcal{R})}(M_{T}, B)\longrightarrow  \Ho_{\mathrm{DgMod}(\mathcal{R})}(M_{T}, B')$. Also remember that the following  diagram  commutes up to the sign $(-1)^{|\varepsilon| |t|}$:
 $$\xymatrix{ \Ho_{\mathrm{DgMod}(\mathcal{R})}(M_{T},B)\ar[rr]_{\Ho_{\mathrm{DgMod}(\mathcal{R})} (M_{T},\varepsilon)} \ar[dd]_{\Ho_{\mathrm{DgMod}(\mathcal{R})} (\overline{t},B)}& &\Ho_{\mathrm{DgMod}(\mathcal{R})}(M_{T},B') \ar[dd]_{\Ho_{\mathrm{DgMod}(\mathcal{R})} (\overline{t},B')} \\
  & & \\
   \Ho_{\mathrm{DgMod}(\mathcal{R})}(M_{T'},B)\ar[rr]_{\Ho_{\mathrm{DgMod}(\mathcal{R})}(M_{T'},\varepsilon)}& & \Ho_{\mathrm{DgMod}(\mathcal{R})}(M_{T'},B').}$$
   for each homogeneous morphism  $t: T \longrightarrow T'$ in  $\T$. As a result $$\mathbb{G}_{1}(\varepsilon)=\big \{ [\mathbb{G}_{1}(\varepsilon)]_{T}:= \Ho_{\mathrm{DgMod}(\mathcal{R})}(M_{T},\varepsilon): \mathbb{G}_{1}(B)(T) \longrightarrow \mathbb{G}_{1}(B')(T) \big \}_{T \in \T}$$
is dg-natural transformation of degree $|\varepsilon|$.
 \end{enumerate}
\end{Remark}

\begin{definition}\label{otrobimodulo}
Let $F:\mathrm{DgMod}(\mathcal{R})\rightarrow \mathrm{DgMod}(\mathcal{S}) $ a dg-functor and $M\in\mathrm{DgMod}(\mathcal{R}\otimes \mathcal{T}^{op})$. We define a dg-bimodule in 
 $\mathrm{DgMod}(\mathcal{S}\otimes\mathcal{T}^{op})$ denoted  by $N:=F(M)$ as follows, the functor $N=F (M) :\mathcal{S}\otimes\mathcal{T}^{op}\rightarrow \mathrm{DgMod}(K)$ is given by:\\
(i) $N(S,T):=F(M_{T})(S)$ for all $(S,T)\in \mathcal{S}\otimes\mathcal{T}^{op}$.\\
 (ii) Let $g\otimes t^{op}:(S,T)\rightarrow (S',T')$ where $g:S\rightarrow S'$ in $\mathcal{S}$ and $t:T'\rightarrow T$ in $\mathcal{T}$ are homogenous. Since $\bar{t}:M_{T}\rightarrow M_{T'}$ is a morphism of $\mathcal{R}$-modules of degree $|t|$, then $F(\bar{t}):F(M_{T})\rightarrow F(M_{T'})$ is a dg-natural transformation of degree $|t|$ in $\mathrm{DgMod}(\mathcal{S})$. Thus we have the following commutative diagram up to the sign $(-1)^{|t||g|}$
      \[
      \begin{diagram}
      \node{F(M_{T})(S)}\arrow{e,t}{[F(\bar{t})]_{S}}\arrow{s,l}{F(M_{T})(g)}
       \node{F(M_{T'})(S)}\arrow{s,r}{F(M_{T'})(g)}\\
      \node{F(M_{T})(S')}\arrow{e,b}{[F(\bar{t})]_{S'}}
       \node{F(M_{T'})(S')}
      \end{diagram}
      \]
Hence we define $N(g\otimes t^{op}):=F(M_{T'})(g)\circ [F(\bar{t})]_{S}=(-1)^{|t||g|}[F(\bar{t})]_{S'}\circ F(M_{T})(g).$
\end{definition}

Now, that we have a dg-bimodule $N\in \mathrm{DgMod}(\mathcal{S}\otimes \mathcal{T}^{op})$ and with this bimodule we define a functor $\mathbb{G}_{2}$ similar to $\mathbb{G}_{1}$. For convenience of the reader  we repeat it's construction.

\begin{Remark}\label{defofG2}
The dg-functor $\mathbb{G}_{2}:\mathrm{DgMod}(\mathcal{S})\longrightarrow \mathrm{DgMod}(\mathcal{T}) $ is defined as follows.
\begin{enumerate}
 \item For  $L \in \mathrm{DgMod}(\mathcal{S})$, $\mathbb{G}_{2}(L)(T)=\Ho_{\mathrm{DgMod}(\mathcal{S})}(N_{T},L)$ for all $T \in \T$, where $N_{T}:=F(M_{T})\in \mathrm{DgMod}(\mathcal{S})$. Now, for  $t: T' \longrightarrow T$ a homogeneous morphism in $\T$ we have $\mathbb{G}_{2}(L)(t):\Ho_{\mathrm{DgMod}(\mathcal{S})}(N_{T'},L) \longrightarrow \Ho_{\mathrm{DgMod}(\mathcal{S})}(N_{T},L)$  and  for each $\eta: N_{T'} \longrightarrow L$  homogeneous morphism of degree $|\eta|$  we have $\mathbb{G}_{2}(L)(t)(\eta):= \Ho_{\mathrm{DgMod}(\mathcal{S})} (F(\overline{t}), L)(\eta)$ $= (-1)^{|\eta||\overline{t}|} \eta \circ F(\overline{t})$.
 
 \item  Let $\gamma: L \longrightarrow L'$ be a  dg-natural transformation of dg $\mathcal{S}$-modules of degree $|\gamma|$. We assert that  $\mathbb{G}_{2}(\gamma):\mathbb{G}_{2}(L) \longrightarrow \mathbb{G}_{2}(L')$ is a dg-natural transformation in $\mathrm{DgMod}(\T)$ of degree $|\gamma|$.\\
Indeed, for each  $T \in \T$ we have that $[\mathbb{G}_{2}(\gamma)]_{T}:= \Ho_{\mathrm{DgMod}(\mathcal{S})}(N_{T}, \gamma):  \Ho_{\mathrm{DgMod}(\mathcal{S})}(N_{T}, L)\longrightarrow  \Ho_{\mathrm{DgMod}(\mathcal{S})}(N_{T}, L')$. Also remember that the following  diagram  commutes up to the sign $(-1)^{|\gamma| |t|}$:
 $$\xymatrix{ \Ho_{\mathrm{DgMod}(\mathcal{S})}(N_{T'},L)\ar[rr]_{\Ho_{\mathrm{DgMod}(\mathcal{S})} (N_{T'},\gamma)} \ar[dd]_{\Ho_{\mathrm{DgMod}(\mathcal{S})} (F(\overline{t}),L)}& &\Ho_{\mathrm{DgMod}(\mathcal{S})}(N_{T'},L') \ar[dd]_{\Ho_{\mathrm{DgMod}(\mathcal{S})} (F(\overline{t}),L')} \\
  & & \\
   \Ho_{\mathrm{DgMod}(\mathcal{S})}(N_{T},L)\ar[rr]_{\Ho_{\mathrm{DgMod}(\mathcal{S})}(N_{T},\gamma)}& & \Ho_{\mathrm{DgMod}(\mathcal{S})}(N_{T},L').}$$
   for each homogeneous morphism  $t: T' \longrightarrow T$ in  $\T$. As a result $$\mathbb{G}_{2}(\gamma)=\big \{ [\mathbb{G}_{2}(\gamma)]_{T}:= \Ho_{\mathrm{DgMod}(\mathcal{S})}(N_{T},\gamma): \mathbb{G}_{2}(L)(T) \longrightarrow \mathbb{G}_{2}(L')(T) \big \}_{T \in \T}$$
is dg-natural transformation of degree $|\gamma|$.
 \end{enumerate}
\end{Remark}

Since $N=F(M) \in \mathrm{DgMod}(\mathcal{S}\otimes \mathcal{T}^{op})$, we construct the dg comma category $ \Big( \mathrm{DgMod}(\mathcal{T}),\mathbb{G}_{2}\mathrm{DgMod}(\mathcal{S})\Big)$, 
and we have an equivalence of dg-categories
$$\Big(\mathrm{DgMod}(\mathcal{T}),\mathbb{G}_{2}\mathrm{DgMod}(\mathcal{S})\Big)\xrightarrow{\sim}\mathrm{DgMod}(\left[ \begin{smallmatrix}
\mathcal{T} & 0 \\
F(M)& \mathcal{S}
\end{smallmatrix} \right] ).$$
\\

Suppose that we have two dg-functors $F:\mathrm{DgMod}(\mathcal{R})\rightarrow \mathrm{DgMod}(\mathcal{S})$ and $G:\mathrm{DgMod}(\mathcal{S})\rightarrow \mathrm{DgMod}(\mathcal{R})$ and a dg-bimodule $M\in\mathrm{DgMod}(\mathcal{R}\otimes \mathcal{T}^{op})$.\\
Then, for all $B\in \mathrm{DgMod}(\mathcal{R})$ and $T\in \mathcal{T}$, $F$ defines a mapping of complexes (that is, a morphism  in $\mathrm{DgMod}_{str}(K)$)
\[
F_{M_{T},B}:\mathrm{Hom}_{\mathrm{DgMod}(\mathcal{R})}(M_{T},B)\longrightarrow \mathrm{Hom}_{\mathrm{DgMod}(\mathcal{S})}(FM_{T},FB), \ f\mapsto F(f).
\] 
Similarly, for all $ L\in \mathrm{DgMod}(\mathcal{S})$ and $T\in \mathcal{T}$, $ G $ defines a mapping of complexes
\[
G_{N_{T},L}:\mathrm{Hom}_{\mathrm{DgMod}(\mathcal{S})}(N_{T},L)\longrightarrow \mathrm{Hom}_{\mathrm{DgMod}(\mathcal{R})}(GN_{T},GL), \ g\mapsto G(g)
\] 
where $N=F(M) \in \mathrm{DgMod}(\mathcal{S}\otimes \mathcal{T}^{op})$ is as given in Definition \ref{otrobimodulo}.\\

In this way we have the following Lemma which is a generalization  \cite[Lemma 4.6]{LGOS2}.

\begin{lemma}\label{RecollH}
Let $M\in \mathrm{DgMod}(\mathcal{R}\otimes \mathcal{T}^{op})$  be. Consider two dg-functors $F:\mathrm{DgMod}(\mathcal{R})\rightarrow \mathrm{DgMod}(\mathcal{S})$ and $G:\mathrm{DgMod}(\mathcal{S})\rightarrow \mathrm{DgMod}(\mathcal{R})$. Then, the following conditions  hold.
\begin{itemize}
\item[(a)] For all $B\in \mathrm{DgMod}(\mathcal{R})$,  the family $F_{M,B}:=\{F_{M_{T},B}:\mathbb{G}_1(B)(T)\rightarrow (\mathbb{G}_2 \circ F)(B)(T)\}_{T\in\mathcal T}$ is dg-natural transformation of degree zero.\\ 
That is, $F_{M,B}:\mathbb{G}_1(B)\rightarrow (\mathbb{G}_2\circ F)(B)$ satisfies the following
$F_{M,B}\in \mathrm{Hom}_{\mathrm{DgMod}(\mathcal{T})}^{0}\Big(\mathbb{G}_1(B),  (\mathbb{G}_2\circ F)(B)\Big)$.
                
\item[(b)]  $ \xi:=F_{M,-}:\mathbb{G}_{1}\longrightarrow \mathbb{G}_{2}\circ F $ is a closed and degree zero dg-natural transformation in $\mathrm{DgFun}\Big(\mathrm{DgMod}(\mathcal{R}), \mathrm{DgMod}(\mathcal{T})\Big)$.
\end{itemize}

\begin{itemize}
\item[(c)] Suppose that $M_{T}=G(N_{T})$ for all $T\in \mathcal{T}$.  Then for all $L\in \mathrm{DgMod}(\mathcal{S}) $,  the family $G_{N,L}:=\{G_{N_{T},L}:\mathbb{G}_2(L)(T)\rightarrow (\mathbb{G}_1 \circ G)(L)(T)\}_{T\in\mathcal T}$ is a dg-natural transformation of degree zero. That is, $G_{N,L}:\mathbb{G}_2(L)\rightarrow (\mathbb{G}_1\circ G)(L)$ satisfies the following
$G_{N,L}\in \mathrm{Hom}_{\mathrm{DgMod}(\mathcal{T})}^{0}\Big(\mathbb{G}_2(L),  (\mathbb{G}_1\circ G)(L)\Big)$.

\item[(d)] $\rho:=G_{N,-}=\mathbb{G}_{2}\longrightarrow \mathbb{G}_{1}\circ G$ is a closed and degree zero dg-natural transformation in  $\mathrm{DgFun}\Big(\mathrm{DgMod}(\mathcal{S}), \mathrm{DgMod}(\mathcal{T})\Big)$.
\end{itemize}
\end{lemma}
\begin{proof}
Since $F(M_{T})=N_{T}$,  for all $ B\in \mathrm{DgMod}(\mathcal{R}) $  and $T\in\mathcal T$  we have $\mathbb{G}_1(B)(T)=\mathrm{Hom}_{\mathrm{DgMod}(\mathcal{R})}(M_{T},B)$ and
 $(\mathbb{G}_2\circ F)(B)(T)=\mathrm{Hom}_{\mathrm{DgMod}(\mathcal{S})}(F(M_{T}),F(B))$.\\
(a) First, let us see that $F_{M,B}:\mathbb{G}_{1}(B)\longrightarrow \mathbb{G}_{2}F(B)$ is  such that for each $T\in \mathcal{T}$ the component
$$F_{M_{T},B}:\mathbb{G}_{1}(B)(T)\longrightarrow \mathbb{G}_{2}F(B)(T)$$
satisfies that $F_{M_{T},B}\in \mathrm{Hom}_{\mathrm{DgMod}(K)}^{0}\Big(\mathbb{G}_{1}(B)(T),\mathbb{G}_{2}F(B)(T)\Big)$. But, this follows since $F_{M_{T},B}$ is a morphism of complexes.\\
Now,  let $ t\in \mathrm{Hom}_{\mathcal{T}}(T,T')$ be a homogenous element and  $B\in \mathrm{DgMod}(S) $. We have to show that the following diagram commutes up to the sign
$(-1)^{0\cdot |t|}=1$:
\[
\begin{diagram}
\node{\mathbb{G}_{1}(B)(T)}\arrow{e,t}{F_{M_{T},B}}\arrow{s,l}{\mathbb{G}_{1}(B)(t)}
 \node{\mathbb{G}_{2}(FB)(T)}\arrow{s,r}{\mathbb{G}_{2}(FB)(t)}\\
\node{\mathbb{G}_{1}(B)(T')}\arrow{e,b}{F_{M_{T'},B}}
 \node{\mathbb{G}_{2}(FB)(T').}
 \end{diagram}
\]
Note that if $ \varphi\in \mathbb{G}_{1}(B)(T)=\mathrm{Hom}_{\mathrm{DgMod}(\mathcal{R})}(M_{T},B)$ is a homogenous element, then
\[
(\mathbb{G}_{2}(FB)(t)\circ F_{M_{T},B})(\varphi)=\mathrm{Hom}_{\mathrm{DgMod}(\mathcal{S})}(F(\overline{t}),FB)F(\varphi)=(-1)^{|\varphi||t|}F(\varphi)F(\overline{t})
\]
and
\begin{align*}
(F_{M_{T'},B}\circ \mathbb{G}_{1}(B)(t))(\varphi) & =F_{M_{T'},B}\left( \mathrm{Hom}_{\mathrm{DgMod}(\mathcal{R})}(\overline{t},B)(\varphi) \right)\\
& =F_{M_{T'},B}((-1)^{|\varphi||t|}\varphi\overline{t})\\
& =(-1)^{|\varphi||t|}F(\varphi\overline{t})\\
& =(-1)^{|\varphi||t|}F(\varphi)F(\overline{t}).
\end{align*}
Proving that the diagram commutes.\\
(b)  By item $(a)$ we have that for each $B\in \mathrm{DgMod}(\mathcal{R})$, we have that
$$F_{M,B}\in \mathrm{Hom}_{\mathrm{DgMod}(\mathcal{T})}^{0}\Big(\mathbb{G}_1(B),  (\mathbb{G}_2\circ F)(B)\Big).$$ This means that
$ \xi:=F_{M,-}:\mathbb{G}_{1}\longrightarrow \mathbb{G}_{2}\circ F $ is of degree zero.\\
Now, let $ f\in \mathrm{Hom}_{\mathrm{DgMod}(\mathcal{R})}(B,B')$ homogenous then we have to show the the following diagram commutes up to the sign $(-1)^{|f|\cdot 0}=1$:
\[
\begin{diagram}
\node{\mathbb{G}_{1}(B)}\arrow{e,t}{F_{M,B}}\arrow{s,l}{\mathbb{G}_{1}(f)}
 \node{\mathbb{G}_{2}F(B)}\arrow{s,r}{\mathbb{G}_{2}F(f)}\\
\node{\mathbb{G}_{1}(B')}\arrow{e,b}{F_{M,B'}}
 \node{\mathbb{G}_{2}F(B')}
 \end{diagram}
\]
For all $ T\in \mathcal{T} $ and for all $ \varphi\in \mathbb{G}_{1}(B)(T) $ we obtain the equalities.
\begin{align*}
([\mathbb{G}_{2}F(f)]_{T}\!\circ \!F_{M_T,B})(\varphi) & =\! [\mathbb{G}_{2}F(f)]_{T}\!\left( F_{M_{T},B}(\varphi)\right)\\
 &=\!\mathrm{Hom}_{\mathrm{DgMod}(\mathcal{S})}(\!F(M_{T}),F(f)\!)F(\varphi)\\
& =F(f)F(\varphi)
\end{align*}
and 
\begin{align*}
([F_{M_T,B'}] \circ[\mathbb{G}_{1}(f)]_{T})(\varphi) =F_{M_{T},B'}\left( \mathrm{Hom}_{\mathrm{DgMod}(\mathcal{R})}(M_{T},f)(\varphi)\right) & =F(f\varphi)\\
& =F(f)F(\varphi).
\end{align*}
Proving that the previous diagram is commutative.\\
This proves that  $ \xi:=F_{M,-}:\mathbb{G}_{1}\longrightarrow \mathbb{G}_{2}\circ F $ is a dg-natural transformation of degree zero. \\
Now, let us see that $\xi$ is closed. Let us consider
$$\delta:\mathrm{DgNat}(\mathbb{G}_{1},\mathbb{G}_{2}\circ F )\longrightarrow \mathrm{DgNat}(\mathbb{G}_{1},\mathbb{G}_{2}\circ F )$$
the differential defined as in Equation \ref{diferentialNat}:

$$\delta(\eta)_{B}:=d(\eta_{B})$$
for all  $\eta\in \mathrm{DgNat}^{n}(\mathbb{G}_{1},\mathbb{G}_{2}\circ F )$ and  for all $B\in \mathrm{DgMod}(\mathcal{R})$, where 
$$d:\mathrm{Hom}_{\mathrm{DgMod}(\mathcal{T})}\Big(\mathbb{G}_{1}(B), \mathbb{G}_{2}F(B)\Big)\longrightarrow \mathrm{Hom}_{\mathrm{DgMod}(\mathcal{T})}\Big(\mathbb{G}_{1}(B), \mathbb{G}_{2}F(B)\Big)$$
is the corresponding differential.\\
We must show that $\delta(\xi)=0$. That is, we have to show that  $d(F_{M,B})=0$ for all  $B\in \mathrm{DgMod}(\mathcal{R})$. By Equation  \ref{dgNatmod}, we have that 
$$d(F_{M,B})_{T}:=\Delta(F_{M_{T},B})$$
 for all $T\in \mathcal{T}$, where the morphism 
 $$\xymatrix{\mathrm{Hom}_{\mathrm{DgMod}(K)}\Big(\mathbb{G}_{1}(B)(T),\mathbb{G}_{2}F(B)(T)\Big)\ar[d]^{\Delta}\\
\mathrm{Hom}_{\mathrm{DgMod}(K)}\Big(\mathbb{G}_{1}(B)(T),\mathbb{G}_{2}F(B)(T)\Big)}$$
is the differential  defined in the Equation  \ref{Homdgestructure}.\\
Since $F_{M_{T},B}:\mathrm{Hom}_{\mathrm{DgMod}(\mathcal{R})}(M_{T},B)\longrightarrow \mathrm{Hom}_{\mathrm{DgMod}(\mathcal{S})}(FM_{T},FB)$, by Equation  \ref{Homdgestructure}, we have that
\begin{align*}
& \Delta(F_{M_{T},B})=\\
& = d_{\mathrm{Hom}_{\mathrm{DgMod}(\mathcal{S})}(FM_{T},FB)}\circ F_{M_{T},B} -(-1)^{|F_{M_{T},B}|}F_{M_{T},B}\circ d_{\mathrm{Hom}_{\mathrm{DgMod}(\mathcal{R})}(M_{T},B)}\\
 & = d_{\mathrm{Hom}_{\mathrm{DgMod}(\mathcal{S})}(FM_{T},FB)}\circ F_{M_{T},B}-F_{M_{T},B}\circ d_{\mathrm{Hom}_{\mathrm{DgMod}(\mathcal{R})}(M_{T},B)}.
\end{align*}
Now, since  $F_{M_{T},B}$ is a morphism of complexes, we have that it commutes with differentials and hence we conclude that $\Delta(F_{M_{T},B})=0$ for all $T\in \mathcal{T}$, and we conclude that $d(F_{M,B})=0$. This proves that $\delta(\xi)=0$, and hence
$\xi:=F_{M,-}:\mathbb{G}_{1}\longrightarrow \mathbb{G}_{2}\circ F $ is a closed and degree zero dg-natural transformation.\\
$(c)$ and $(d)$. Suppose that $M_{T}=G(N_{T})$ for all $T\in \mathcal{T}$. Then, for all  $L\in \mathrm{DgMod}(\mathcal{S})$ and $T\in\mathcal T$ 
we have $\mathbb{G}_2(L)(T)=\mathrm{Hom}_{\mathrm{DgMod}(\mathcal{S})}(N_{T},L)$ and
 $(\mathbb{G}_1\circ G) (L)(T)=\mathrm{Hom}_{\mathrm{DgMod}(\mathcal{R})}(G(N_{T}),G(L))$. Therefore the proof of $(c)$ and $(d)$  is similar to $(a)$ and $(b)$.
\end{proof}

We have the following definition which is a generalization of the $\textnormal{\cite[Definition 3.2]{Chen}}$.
\begin{definition}
Let $G_{1}: \mathcal{A}\rightarrow  \mathcal{D}$, $G_{2}: \mathcal {B}\rightarrow  \mathcal{D}$,  $F: \mathcal{A}\rightarrow  \mathcal{B}$ and 
$H: \mathcal{B}\rightarrow  \mathcal{A}$ dg-functors. Assume that $(F,H)$ is a dg-adjoint pair, with $\eta_{X,Y}:\mathrm{Hom}_{\mathcal{B}}(FX,Y)\longrightarrow \mathrm{Hom}_{\mathcal{A}}(X,HY)$ being the adjugant equivalence $\forall X\in \mathcal{A}$  and $\forall Y\in \mathcal{B}$. We say that the pair $(G_{1},G_{2})$ is $\textbf{dg-compatible}$ with the dg-adjoint pair $(F,H)$ if there exist two closed and degree zero dg-natural transformations
$$ \xi:G_{1} \rightarrow G_{2}, \quad  \rho:G_{2}\rightarrow G_{1}H$$ 
such that $\rho_{Y}$ is a monomorphism and
\begin{equation}\label{condicioncompa}
G_{1}(\eta_{X,Y}(f))=\rho_{Y}\circ G_2(f)\circ \xi_X
\end{equation}
 for every $X\in  \mathcal{A}$, $Y\in  \mathcal{B}$ and $f\in\mathrm{Hom}_{\mathcal{B}}(FX,Y)$.
\end{definition}

We recall the following Definition given in \cite{PaezSandovalSan}.

\begin{definition}\cite[Definition 3.1]{PaezSandovalSan}\label{commadgcat}
Let $\mathcal{C}$ and $\mathcal{D}$ two dg-categories and $F:\mathcal{C}\longrightarrow \mathcal{D}$ be a dg-functor. We define the $\textbf{comma category}$ $\big(\mathcal{D},F(\mathcal{C})\big)$. 
\begin{enumerate}
\item [(a)] The objects of $\big(\mathcal{D},F(\mathcal{C})\big)$ are morphisms $f:D\longrightarrow F(C)$ in $\mathcal{D}$ with $f\in \mathrm{Hom}_{\mathcal{D}}^{0}(D,F(C))$ and $d_{\mathcal{D}(D,F(C))}(f)=0$. The object
$f:D\longrightarrow F(C)$ will be denoted by $(D,f,C)$.

\item [(b)] Given two objects $f:D\longrightarrow F(C)$ and $f:D'\longrightarrow F(C')$ in $\big(\mathcal{D},F(\mathcal{C})\big)$ a morphism from $f$ to $f'$ is a pair of morphisms $\alpha:D\longrightarrow D'$ and $\beta:C\longrightarrow C'$ such that the following diagram commutes in $\mathcal{D}$:
$$\xymatrix{D\ar[r]^{\alpha}\ar[d]_{f} & D'\ar[d]^{f'}\\
F(C)\ar[r]^{F(\beta)} & F(C').}$$
\end{enumerate}
The composition is given as follows: let $(\alpha,\beta):(D,f,C)\longrightarrow (D',f',C')$ and $(\alpha',\beta'):(D',f',C')\longrightarrow (D'',f'',C'')$ morphisms in  $\big(\mathcal{D},F(\mathcal{C})\big)$ we set:
$$(\alpha',\beta')\circ (\alpha,\beta):=(\alpha'\circ \alpha,\beta'\circ\beta).$$
\end{definition}

\begin{proposition}\cite[Proposition 3.2]{PaezSandovalSan}\label{dgcommaprop}
Let $\mathcal{C}$ and $\mathcal{D}$ two dg-categories and $F:\mathcal{C}\longrightarrow \mathcal{D}$ be a dg-functor. Then the comma category $\big(\mathcal{D},F(\mathcal{C})\big)$ is a dg-category.
\end{proposition}

Now, we have the following result that is analogous to \cite[Theorem 3.6]{Chen}.
\begin{theorem}\label{DgRecollJ}
Let $\mathcal{A}$, $\mathcal{B}$, $\mathcal{C}$ and $\mathcal{D}$ dg-categories, and let  $G_{1}:\mathcal{A}\longrightarrow \mathcal{D}$ and 
$G_{2}:\mathcal{B}\longrightarrow \mathcal{D}$ dg-functors.
\begin{enumerate}
\item [(a)] If the diagram 
$$\xymatrix{\mathcal{C}\ar@<-2ex>[rr]_{i_{\ast}} & & \mathcal{A}\ar@<-2ex>[rr]_{j^{!}}\ar@<-2ex>[ll]_{i^{\ast}} & &  \mathcal{B}\ar@<-2ex>[ll]_{j_{!}}}$$
is a left dg-recollement where, $(G_{2},G_{1})$ is dg-compatible with the dg-adjoint pair $(j_{!},j^{!})$,  then there is a left dg-recollement
$$\xymatrix{\mathcal{C}\ar@<-2ex>[rr]_{\tilde{i_{\ast}}} & & \Big(\mathcal{D},G_{1}\mathcal{A}\Big)\ar@<-2ex>[rr]_{\tilde{j^{!}}}\ar@<-2ex>[ll]_{\tilde{i^{\ast}}} & &  \Big(\mathcal{D},G_{2}\mathcal{B}\Big)\ar@<-2ex>[ll]_{\tilde{j_{!}}}}$$

\item [(b)] If the diagram 
$$\xymatrix{\mathcal{C}\ar@<2ex>[rr]^{i_{!}} & & \mathcal{A}\ar@<2ex>[rr]^{j^{\ast}}\ar@<2ex>[ll]^{i^{!}} & &  \mathcal{B}\ar@<2ex>[ll]^{j_{\ast}}}$$
is a right dg-recollement, where $(G_{1},G_{2})$ is dg-compatible with the dg-adjoint pair $(j^{\ast},j_{\ast})$, then there is a right dg-recollement
$$\xymatrix{\mathcal{C}\ar@<2ex>[rr]^{\tilde{i_{!}}} & & \Big(\mathcal{D},G_{1}\mathcal{A}\Big)\ar@<2ex>[rr]^{\tilde{j^{\ast}}}\ar@<2ex>[ll]^{\tilde{i^{!}}} & & \Big(\mathcal{D},G_{2}\mathcal{B}\Big)\ar@<2ex>[ll]^{\tilde{j_{\ast}}}}$$
\end{enumerate}
\end{theorem}
\begin{proof}
The proof given in \cite[Theorem 3.6]{Chen} can be adapted to this setting.
\end{proof}

In order to prove  the main Theorem of this section we need the following result which generalizes \cite[Lemma 4.2]{Chen} and \cite[Lemma 4.9]{LGOS2}.

\begin{lemma}\label{RecollG}
Let $ \mathcal{R},\, \mathcal{S}  $ and $ \mathcal{T} $ be dg-categories, $ F:\mathrm{DgMod}(\mathcal{R})\longrightarrow \mathrm{DgMod}(\mathcal{S}) $ and $G:\mathrm{DgMod}(\mathcal{S}) \longrightarrow \mathrm{DgMod}(\mathcal{R})$ be dg-functors. For $ M\in \mathrm{DgMod}(\mathcal{R}\otimes \mathcal{T}^{op})$ consider the dg-functors $\mathbb{G}_{1}:\mathrm{DgMod}(\mathcal{R})\longrightarrow \mathrm{DgMod}(\mathcal{T}) $ and $ \mathbb{G}_{2}:\mathrm{DgMod}(\mathcal{S})\longrightarrow \mathrm{DgMod}(\mathcal{T}) $ as we have defined in Remarks \ref{defofG1} and \ref{defofG2} where $N_{T}=F(M_{T} )$. If $ (F,G) $ is a dg-adjoint pair and its unit $ \varepsilon:1_{\mathrm{DgMod}(\mathcal{R})}\longrightarrow GF $ satisfies $ \varepsilon_{M_{T}}=1_{M_{T}}$  for all $T\in \mathcal{T}$, then the pair $ (\mathbb{G}_{1},\mathbb{G}_{2})$ is dg-compatible with $(F,G)$.
\end{lemma}
\begin{proof}
Since $ (F,G) $ is dg-adjoint pair there exist a natural equivalence 
\[
\eta:=\!\!\lbrace \eta_{B,L}\!:\!\mathrm{Hom}_{\mathrm{DgMod}(\mathcal{S})}(FB,L)\rightarrow \mathrm{Hom}_{\mathrm{DgMod}(\mathcal{R})}(B,GL)\rbrace
\]
for $B\in \mathrm{DgMod}(\mathcal{R})$ and $L\in \mathrm{DgMod}(\mathcal{S})$,
where  each $\eta_{B,L}$ is an isomorphism of complexes that is natural in $B$ and $L$. We note that $\eta_{B,L}$ is an isomorphism in $\mathrm{DgMod}(K)$. \\
First recall that in this case $\varepsilon$  is a closed and degree zero dg-natural transformation.\\
By Lemma \ref{RecollH} (b), we have a closed and degree zero dg-natural transformation $ \xi:=F_{M,-}:\mathbb{G}_{1}\longrightarrow \mathbb{G}_{2}F $. Since $ \varepsilon_{M_{T}}=1_{M_{T}}$, we have that $G(N_{T})=M_{T}$ for all $T\in \mathcal{T}$ and by Lemma \ref{RecollH} (d), we  have a closed and degree zero dg-natural transformation $\rho:=G_{N,-}=\mathbb{G}_{2}\longrightarrow \mathbb{G}_{1}G$.\\
First we will see that for all $L\in \mathrm{DgMod}(\mathcal S)$ the morphism 
$$\rho_{L}:\mathbb{G}_{2}(L)\longrightarrow \mathbb{G}_{1}G(L)$$
is a monomorphism in $\mathrm{DgMod}(\mathcal{T})$. Indeed, for $T\in \mathcal{T}$ we have to show that 
$$[\rho_{L}]_{T}:=G_{N_{T},L}:\mathrm{Hom}_{\mathrm{DgMod}(\mathcal{S})}(N_{T},L)\rightarrow \mathrm{Hom}_{\mathrm{DgMod}(\mathcal {R})}(M_{T},G(L))$$
is a monomorphism. Consider the isomorphism of complexes (given by $\eta$)
$$\eta_{M_{T},L}:\mathrm{Hom}_{\mathrm{DgMod}(\mathcal{S})}(F(M_{T}),L)\longrightarrow \mathrm{Hom}_{\mathrm{DgMod}(\mathcal{R})}(M_{T},G(L)).$$
We assert that $\eta_{M_{T},L}=G_{N_{T},L}.$ Indeed, let 
$\beta \in  \mathrm{Hom}_{\mathrm{DgMod}(\mathcal{S})}(F(M_{T}),L)$ a homogenous element. Since $\eta_{M_{T},L}$ is natural in $L$, we have the following commutative diagram in $\mathrm{DgMod}(K)$:
$$\xymatrix{\mathrm{Hom}_{\mathrm{DgMod}(\mathcal{S})}(
F(M_{T}),F(M_{T}))\ar[rr]^{\eta_{M_{T},F(M_{T})}}\ar[d]_{\mathrm{Hom}_{\mathrm{DgMod}(\mathcal{S})}(F(M_{T}),\beta)} & & \mathrm{Hom}_{\mathrm{DgMod}(\mathcal{R})}(M_{T},GF(M_{T}))\ar[d]_{\mathrm{Hom}_{\mathrm{DgMod}(\mathcal{R})}(\!M_{T},G(\beta)\!)}\\
\mathrm{Hom}_{\mathrm{DgMod}(\mathcal{S})}(F(M_{T}),L)\ar[rr]_{\eta_{M_{T},L}} & & \mathrm{Hom}_{\mathrm{DgMod}(\mathcal{R})}(M_{T},G(L))}$$
Then, for $1_{F(M_{T})}:F(M_{T})\longrightarrow F(M_{T})$ we have that:

\begin{align*}
G(\beta)\!\circ\! \big(\! \eta_{M_{T},F(M_{T})}(1_{F(M_{T})})\!\big)& \!=\! \big(\!\mathrm{Hom}_{\mathrm{DgMod}(\mathcal{R})}(M_{T},G(\beta))\!\circ\!\eta_{M_{T},F(M_{T})}\!\big)(1_{F(M_{T})})\\
& \!=\! \big(\eta_{M_{T},L}\circ \mathrm{Hom}_{\mathrm{DgMod}(\mathcal{S})}(N_{T},\beta)\big)(1_{F(M_{T})})\\
& = \eta_{M_{T},L}\Big(\mathrm{Hom}_{\mathrm{DgMod}(\mathcal{S})}(N_{T},\beta)\big)(1_{F(M_{T})})\Big)\\
&=\eta_{M_{T},L}(\beta\circ 1_{F(M_{T})})\\
&=  \eta_{M_{T},L}(\beta)
\end{align*}
Since $1_{M_{T}}=\varepsilon_{M_{T}}= \eta_{M_{T},F(M_{T})}(1_{F(M_{T})})$ then $ G(\beta)= \eta_{M_{T},L}(\beta)$.
Since $G_{N_{T},L}(\beta)=G(\beta)$ and $\eta_{M_{T},L}$ is an isomorphism (in particular it is a monomorphism) it follows that $G_{N_{T},L}$ is a monomorphism in $\mathrm{DgMod}(K)$, for all $ T\in \mathcal{T}$.
Hence, $\rho_{L}$ is a monomorphism for each $L\in \mathrm{DgMod}(\mathcal{S})$.\\
Now we have to show that $\mathbb{G}_1(\eta_{B,L}(f))=\rho_{L}\mathbb{G}_2(f)\xi_B,\,\,\forall f\in \mathrm{Hom}_{\mathrm{DgMod}(\mathcal {S})}(F(B),L)$. It is enough to see the equality in homogenous elements $f\in \mathrm{Hom}_{\mathrm{DgMod}(\mathcal {S})}(F(B),L)$.\\
We have to show that
$$
[\mathbb{G}_{1}\left( \eta_{B,L}(f)\right)]_{T}=G_{N_{T},L} [\mathbb{G}_2(f)]_{T}F_{M_{T},B}\,\,\,\,\,\,\forall T\in\mathcal T.$$ 

Let  $ \alpha\in \mathrm{Hom}_{\mathrm{DgMod}(\mathcal{R})}(M_{T},B)$ be a homogenous element.  Since $\eta_{B,L}$ is natural in $B$ we have the  following commutative diagram
\[
\begin{diagram}
\node{\mathrm{Hom}_{\mathrm{DgMod}(\mathcal{S})}(F(B),L)}\arrow{e,t}{\eta_{B,L}}\arrow{s,l}{\mathrm{Hom}_{\mathrm{DgMod}(\mathcal{S})}(F(\alpha),L)}
 \node{\mathrm{Hom}_{\mathrm{DgMod}(\mathcal{R})}(B,G(L))}\arrow{s,r}{\mathrm{Hom}_{\mathrm{DgMod}(\mathcal{R})}(\alpha,G(L))}\\
\node{\mathrm{Hom}_{\mathrm{DgMod}(\mathcal{S})}(F(M_{T}),L)}\arrow{e,b}{\eta_{M_{T},L}}
 \node{\mathrm{Hom}_{\mathrm{DgMod}(\mathcal{R})}(M_{T},G(L)).}
\end{diagram}
\]
Then, since $|F(\alpha)|=|\alpha|$ and $|\eta_{B,L}(f)|=|f|$ we have that 
\begin{align*}
 (-1)^{|f||\alpha|} G\Big(f\circ F(\alpha)\Big) & = G\Big((-1)^{|f||\alpha|}f\circ F(\alpha)\Big)\\
& =\eta_{M_{T},L}\Big((-1)^{|f||\alpha|}f\circ F(\alpha)\Big)\\
& =\eta_{M_{T},L}\Big(\mathrm{Hom}_{\mathrm{DgMod}(\mathcal{S})}(F(\alpha),L)(f)\Big)\\
& = \mathrm{Hom}_{\mathrm{DgMod}(\mathcal{R})}(\alpha,G(L))\Big(\eta_{B,L}(f)\Big)\\
& = (-1)^{|f||\alpha|}\eta_{B,L}(f)\circ \alpha.
\end{align*}
Therefore, $G\Big(f\circ F(\alpha)\Big)=\eta_{B,L}(f)\circ \alpha.$ We note that
\[
\Big[\mathbb{G}_{1}\left( \eta_{B,L}(f)\right)\Big]_{T}(\alpha)=\mathrm{Hom}_{\mathrm{DgMod}(\mathcal{R})}\Big(M_{T},\eta_{B,L}(f)\Big)(\alpha)=\eta_{B,L}(f)\circ \alpha.
\]
On the other hand,
\begin{align*}
\Big(\!G_{N_T,L}\!\circ\! [\mathbb{G}_{2}(f)]_{T}\!\circ \!F_{M_{T},B}\!\Big)(\alpha)&= (G_{N_{T},L}\!\circ\!\mathrm{Hom}_{\mathrm{DgMod}(\mathcal{S})}(F(M_{T}),f))( F_{M_{T},B}(\alpha))\\
&= (G_{N_{T},L}\circ\mathrm{Hom}_{\mathrm{DgMod}(\mathcal{S})}(F(M_{T}),f))(F(\alpha)) \\
&= G_{N_{T},L}(f\circ F(\alpha)) \\
&= G( f\circ F(\alpha)).
\end{align*}
This proves $\mathbb{G}_1(\eta_{B,L}(f))=\rho_{L}\mathbb{G}_2(f)\xi_B,\,\,\forall f\in \mathrm{Hom}_{\mathrm{DgMod}(\mathcal {S})}(F(B),L)$, and hence $(\mathbb{G}_{1},\mathbb{G}_{2})$ is compatible with $ (F,G) $.
\end{proof}

\begin{lemma}\label{dgcambiofun}
Let $F$ be a dg-fully faithful functor. If $(F,G)$ is a dg-adjoint pair, then there exists a dg-functor $G'$ such that $(F,G')$ is a dg-adjoint pair, where $G\simeq G'$ and the unit $\epsilon':1\longrightarrow G'F$ is the identity.
\end{lemma}
\begin{proof}
The proof given in \cite[Proposition 7.6]{Hilton} in p. 67, can be adapted to this setting.
\end{proof}

\begin{theorem}\label{Recoll1}
Let $\mathcal{R}$, $\mathcal{S}$, $\mathcal{C}$ and $\mathcal{T}$ be dg-categories and $M\in\mathrm{DgMod}(\mathcal{R}\otimes \mathcal{T}^{op})$.
\begin{enumerate}
\item [(a)]
If the diagram 
$$\xymatrix{\mathrm{DgMod}(\mathcal{C})\ar@<-2ex>[rr]_{i_{\ast}} & & \mathrm{DgMod}(\mathcal{S})\ar@<-2ex>[rr]_{j^{!}}\ar@<-2ex>[ll]_{i^{\ast}} & &  \mathrm{DgMod}(\mathcal{R})\ar@<-2ex>[ll]_{j_{!}}}$$
is a left dg-recollement, then there is a left dg-recollement
$$\xymatrix{\mathrm{DgMod}(\mathcal{C})\ar@<-2ex>[rr]_{\tilde{i_{\ast}}} & & \mathrm{DgMod}(\mathbf{\Lambda}^{!})\ar@<-2ex>[rr]_{\tilde{j^{!}}}\ar@<-2ex>[ll]_{\tilde{i^{\ast}}} & &  \mathrm{DgMod}(\mathbf{\Lambda})\ar@<-2ex>[ll]_{\tilde{j_{!}}}}$$
 where $\mathbf{\Lambda}:=\left(\begin{smallmatrix}
\mathcal{T} &0\\
M& \mathcal{R}
\end{smallmatrix}\right )$ and $\mathbf{\Lambda}^!:=\left (\begin{smallmatrix}
\mathcal{T} &0\\
j_{!}(M)& \mathcal{S}
\end{smallmatrix}\right )$ are  dg-matrix categories and  $j_{!}(M)$  is the dg-bimodule constructed as in Definition \ref{otrobimodulo}.

\item [(b)] If the diagram 
$$\xymatrix{\mathrm{DgMod}(\mathcal{C})\ar@<2ex>[rr]^{i_{!}} & & \mathrm{DgMod}(\mathcal{S})\ar@<2ex>[rr]^{j^{\ast}}\ar@<2ex>[ll]^{i^{!}} & &  \mathrm{DgMod}(\mathcal{R})\ar@<2ex>[ll]^{j_{\ast}}}$$
is a right recollement, then there is a right recollement
$$\xymatrix{\mathrm{DgMod}(\mathcal{C})\ar@<2ex>[rr]^{\tilde{i_{!}}} & & \mathrm{DgMod}(\mathbf{\Lambda}^{\ast})\ar@<2ex>[rr]^{\tilde{j^{\ast}}}\ar@<2ex>[ll]^{\tilde{i^{!}}} & & \mathrm{DgMod}(\mathbf{\Lambda})\ar@<2ex>[ll]^{\tilde{j_{\ast}}}}$$
where $\mathbf{\Lambda}:=\left(\begin{smallmatrix}
\mathcal{T} &0\\
M& \mathcal{R}
\end{smallmatrix}\right )$ and $\mathbf{\Lambda}^*:=\left (\begin{smallmatrix}
\mathcal{T} &0\\
j_{*}(M)& \mathcal{S}
\end{smallmatrix}\right )$ are dg-matrix categories and $j_{\ast}(M)$ is the dg-bimodule constructed as in Definition \ref{otrobimodulo}.
\end{enumerate}
\end{theorem}
\begin{proof}
We only prove (a), since (b) is similar. Consider $$\xymatrix{\mathrm{DgMod}(\mathcal{C})\ar@<-2ex>[rr]_{i_{\ast}} & & \mathrm{DgMod}(\mathcal{S})\ar@<-2ex>[rr]_{j^{!}}\ar@<-2ex>[ll]_{i^{\ast}} & &  \mathrm{DgMod}(\mathcal{R})\ar@<-2ex>[ll]_{j_{!}}}$$ a left dg-recollement and $M\in \mathrm{DgMod}(\mathcal{R}\otimes \mathcal{T}^{op})$. Set $N=j_{!}(M)$ as in Definition \ref{otrobimodulo}, and consider the dg-functors $\mathbb{G}_{1}:\mathrm{DgMod}(\mathcal{R})\rightarrow \mathrm{DgMod}(\mathcal{T})$ and
$\mathbb{G}_{2}:\mathrm{DgMod}(\mathcal{S})\rightarrow \mathrm{DgMod}(\mathcal{T})$ as defined in Remarks \ref{defofG1} and \ref{defofG2}. Since $j_{!}$ is dg-fully faithful,
 by Lemma \ref{dgcambiofun}, we may assume that the unit
$\epsilon:1\rightarrow j^{!}j_{!}$, of the  dg-adjoint pair $(j_{!},j^{!})$, is the identity. In particular, we have that $\epsilon_{M_T}=1_{M_T}$. Thus, from Lemma \ref{RecollG}, the pair $(\mathbb{G}_1,\mathbb{G}_2)$ is 
compatible with $(j_{!},j^{!})$  and the rest follows from Theorems \ref{DgRecollJ} and \ref{equivalenceLEOS}.
\end{proof}


\footnotesize

\vskip3mm \noindent Martha Lizbeth Shaid Sandoval Miranda:\\ Departamento de Matem\'aticas, Universidad Aut\'onoma Metropolitana Unidad Iztapalapa\\
Av. San Rafael Atlixco 186, Col. Vicentina Iztapalapa 09340, M\'exico, Ciudad de M\'exico.\\ {\tt marlisha@xanum.uam.mx, marlisha@ciencias.unam.mx}

\vskip3mm \noindent Valente Santiago Vargas:\\ Departamento de Matem\'aticas, Facultad de Ciencias, Universidad Nacional Aut\'onoma de M\'exico\\
Circuito Exterior, Ciudad Universitaria,
C.P. 04510, Ciudad de M\'exico, MEXICO.\\ {\tt valente.santiago@ciencias.unam.mx}

\vskip3mm \noindent  Edgar Omar Velasco P\'aez:\\ Departamento de Matem\'aticas, Facultad de Ciencias, Universidad Nacional Aut\'onoma de M\'exico\\
Circuito Exterior, Ciudad Universitaria,
C.P. 04510, Ciudad de M\'exico, MEXICO.\\ {\tt edgar-bkz13@ciencias.unam.mx }

\end{document}